\documentclass[11pt]{amsart}

\usepackage[T1]{fontenc}
\IfFileExists{lmodern.sty}{%
  \usepackage{lmodern}%
  \IfFileExists{microtype.sty}{\usepackage{microtype}}{}%
}{%
  \IfFileExists{microtype.sty}{\usepackage[expansion=false]{microtype}}{}%
}
\usepackage{amsmath,amssymb,amsthm,mathtools,mathrsfs}
\usepackage{enumitem}
\usepackage{aliascnt}
\usepackage{tikz}
\usetikzlibrary{arrows.meta,calc}
\usepackage{booktabs}
\usepackage{caption}
\usepackage{hyperref}
\usepackage[nameinlink,noabbrev]{cleveref}

\hypersetup{
  colorlinks=true,
  linkcolor=blue,
  citecolor=blue,
  urlcolor=blue,
  pdftitle={Ropelength-Filtered Knot Spaces},
  pdfauthor={Makoto Ozawa}
}

\numberwithin{equation}{section}
\theoremstyle{plain}
\newtheorem{theorem}{Theorem}[section]

\newaliascnt{proposition}{theorem}
\newtheorem{proposition}[proposition]{Proposition}
\aliascntresetthe{proposition}

\newaliascnt{corollary}{theorem}
\newtheorem{corollary}[corollary]{Corollary}
\aliascntresetthe{corollary}

\newaliascnt{lemma}{theorem}
\newtheorem{lemma}[lemma]{Lemma}
\aliascntresetthe{lemma}

\newaliascnt{conjecture}{theorem}

\aliascntresetthe{conjecture}

\theoremstyle{definition}
\newaliascnt{definition}{theorem}
\newtheorem{definition}[definition]{Definition}
\aliascntresetthe{definition}

\newaliascnt{example}{theorem}

\aliascntresetthe{example}

\newaliascnt{problem}{theorem}
\newtheorem{problem}[problem]{Problem}
\aliascntresetthe{problem}

\newaliascnt{question}{theorem}
\newtheorem{question}[question]{Question}
\aliascntresetthe{question}

\theoremstyle{remark}
\newaliascnt{remark}{theorem}
\newtheorem{remark}[remark]{Remark}
\aliascntresetthe{remark}

\newaliascnt{convention}{theorem}
\newtheorem{convention}[convention]{Convention}
\aliascntresetthe{convention}

\crefname{theorem}{theorem}{theorems}
\Crefname{theorem}{Theorem}{Theorems}
\crefname{proposition}{proposition}{propositions}
\Crefname{proposition}{Proposition}{Propositions}
\crefname{corollary}{corollary}{corollaries}
\Crefname{corollary}{Corollary}{Corollaries}
\crefname{lemma}{lemma}{lemmas}
\Crefname{lemma}{Lemma}{Lemmas}
\crefname{definition}{definition}{definitions}
\Crefname{definition}{Definition}{Definitions}
\crefname{remark}{remark}{remarks}
\Crefname{remark}{Remark}{Remarks}
\crefname{example}{example}{examples}
\Crefname{example}{Example}{Examples}
\crefname{problem}{problem}{problems}
\Crefname{problem}{Problem}{Problems}
\crefname{question}{question}{questions}
\Crefname{question}{Question}{Questions}
\crefname{conjecture}{conjecture}{conjectures}
\Crefname{conjecture}{Conjecture}{Conjectures}
\crefname{convention}{convention}{conventions}
\Crefname{convention}{Convention}{Conventions}
\crefname{figure}{Figure}{Figures}
\Crefname{figure}{Figure}{Figures}
\crefname{table}{Table}{Tables}
\Crefname{table}{Table}{Tables}

\newcommand{\R}{\mathbb{R}}
\newcommand{\Z}{\mathbb{Z}}

\newcommand{\Emb}{\operatorname{Emb}}
\newcommand{\Diff}{\operatorname{Diff}}
\newcommand{\Isom}{\operatorname{Isom}}
\newcommand{\Len}{\operatorname{Len}}
\newcommand{\Thi}{\operatorname{Thi}}
\newcommand{\Rop}{\operatorname{Rop}}
\newcommand{\hocolim}{\operatorname*{hocolim}}
\newcommand{\colim}{\operatorname*{colim}}
\newcommand{\Conn}{\operatorname{Conn}}
\newcommand{\im}{\operatorname{im}}
\newcommand{\id}{\operatorname{id}}

\newcommand{\calE}{\mathcal{E}}
\newcommand{\calK}{\mathcal{K}}
\newcommand{\calP}{\mathcal{P}}
\newcommand{\calR}{\mathcal{R}}
\newcommand{\calU}{\mathcal{U}}

\newcommand{\eps}{\varepsilon}
\newcommand{\wh}[1]{\widehat{#1}}
\newcommand{\wid}{\mathfrak{w}}
\newcommand{\exc}{\ell}

\definecolor{ropeblue}{RGB}{31,86,140}
\definecolor{ropered}{RGB}{178,45,45}
\definecolor{ropegreen}{RGB}{40,110,70}
\definecolor{ropegrey}{RGB}{120,120,120}
\definecolor{ropelight}{RGB}{222,233,243}
\definecolor{ropesand}{RGB}{246,238,222}

\title[Ropelength-filtered knot spaces]
{Ropelength-Filtered Knot Spaces}

\author{Makoto Ozawa}
\address{Department of Natural Sciences, Faculty of Arts and Sciences,
Komazawa University, 1-23-1 Komazawa, Setagaya-ku, Tokyo, 154-8525, Japan}
\email{w3c@komazawa-u.ac.jp}

\date{September 2026}


\begin{document}

\begin{abstract}
Ropelength is usually studied as a minimisation problem for one knot at a
time.  Here we instead use ropelength to filter the space of all realisations
of a knot type.  For a knot type \(K\) and budget \(\Lambda\), let
\(Y_\Lambda(K)\) be the space of unit-thickness \(C^{1,1}\) configurations
of length at most \(\Lambda\), modulo rigid motions and constant-speed
reparametrisation.

We prove compact capture for smooth knot families, yielding homotopical and
homological exhaustion of the ordinary knot space by finite ropelength
levels.  In the normalised Euclidean model, every \(Y_\Lambda(K)\) is compact;
connected components are right-continuous and their merge levels are
attained, giving an ultrametric on the components of the ideal stratum.  We
also prove that \(Y_L(K)\) strongly deformation retracts onto its exact-length
shell, so path connectivity at a fixed level is fixed-length physical
isotopy.  The connected merge scale and the path merge scale, the min--max
quantity for constrained deformation paths, take the same values, although
only the former is known to be attained.
Moreover, all configurations in any fixed finite sublevel become path
connected to one another at some common finite higher level, even if the
original sublevel has infinitely many components.

For higher homotopy we introduce ropelength widths and relate them to known
models of spaces of knots.  We record mirror symmetry and prove finite factorization across every positive filtration gap in a marked weak model, yielding all-degree homological q-tameness.  Finally,
we relate the path filtration to the companion swept-area construction: the infimal budget for admissible endpoint connections agrees with the path merge threshold.  The resulting framework separates geometric,
topological, and physical-isotopy questions and provides a quantitative
foundation for the study of ropelength-filtered knot spaces.
\end{abstract}

\maketitle
\newpage
\tableofcontents

\section{Introduction}\label{sec:introduction}

\subsection{A rope, a knot, and a budget}

Take a piece of rope of diameter \(2\), tie a knot in it, and splice the
ends.  You now have a closed curve together with an embedded tube of radius
\(1\) around it.  Two numbers describe the result: the length \(\Len\) of
the core curve, and the largest radius \(\Thi\) of an embedded tube, called
the \emph{thickness}.  Their quotient
\[
  \Rop(\gamma)=\frac{\Len(\gamma)}{\Thi(\gamma)}
\]
is the \emph{ropelength}: the dimensionless ratio of core length to
tube radius.  It does not change if the whole picture is
rescaled, so it is an invariant of shape rather than of size.

Now ask a physical question.

\begin{question}[The rope-budget question]\label{q:budget}
Given two configurations of the same knot type made from the same rope, can
one be deformed into the other \emph{without ever cutting the rope, letting
it pass through itself, or using more rope than we started with}?
\end{question}

Classical knot theory answers a weaker question: two configurations of the
same knot type are ambient isotopic.  But an ambient isotopy is allowed to
squeeze the rope to zero thickness, or to borrow arbitrarily much extra
length in the middle of the motion, and then give it back.  Whether a
motion exists \emph{within a fixed geometric budget} is a different
question, and the answer is different.

The most famous instance is the \emph{Gordian unknot}
(\cref{fig:gordian}).  Pieranski, Przyby\l{} and Stasiak
\cite{PieranskiGordian} exhibited numerical configurations of the unknot
whose tightening flow appears to become trapped.  The terminology now
requires some care because several closely related ``physical'' models are
in use.  In the standard thick-rope model used in this paper, the same
radius controls both the embedded tube and the curvature obstruction.  For
links, Coward and Hass \cite{CowardHass} proved that physical and topological
isotopy can differ, and Kusner and Kusner \cite{KusnerKusner} found a
Gordian pair even when only total length is constrained.  More recently,
Ayala and Hass \cite{AyalaHass} constructed large families of Gordian
unlinks, and Bauermeister \cite{Bauermeister} obtained further split-link
phenomena in the more permissive Gehring ropelength setting.  For a
\emph{single unknot}, Ayala \cite{AyalaGeometric} proves Gordian behaviour
in the thinner curvature-constrained model \(\mathcal U_1\).  In the
notation of that work, the standard thick model adopted here corresponds to
\(\mathcal U_2\), so the two questions should be distinguished.  This
model-dependence is not a technicality: it is one of the main reasons for
formulating the spaces and the allowed paths explicitly.

\begin{figure}[t]
\centering
\begin{tikzpicture}[scale=0.95]
% --- left: a rope with a tube
\begin{scope}[shift={(0,0)}]
  \draw[line width=14pt,ropelight,line cap=round]
     (-1.9,0) .. controls (-1.2,0.9) and (-0.4,-0.9) .. (0.35,0.1)
              .. controls (0.9,0.85) and (1.5,-0.5) .. (2.0,0.15);
  \draw[line width=1.1pt,ropeblue,line cap=round]
     (-1.9,0) .. controls (-1.2,0.9) and (-0.4,-0.9) .. (0.35,0.1)
              .. controls (0.9,0.85) and (1.5,-0.5) .. (2.0,0.15);
  \draw[->,>=Stealth,ropered,line width=0.7pt] (0.35,0.1) -- ++(0,0.245);
  \filldraw[ropered] (0.35,0.1) circle (1.1pt);
  \node[ropered,font=\scriptsize,above] at (0.35,0.4) {$\Thi$};
  \node[font=\scriptsize,below] at (0.05,-0.8) {core curve inside its tube};
\end{scope}
% --- right: a tangled unknot (three-crossing unknot diagram)
\begin{scope}[shift={(5.7,0.05)}]
\draw[line width=6pt,ropesand,line join=round] plot[domain=0:360,samples=181,variable=\t] ({0.36*(sin(\t)+2*sin(2*\t))},{0.36*(cos(\t)-2*cos(2*\t))}) -- cycle;
\draw[line width=1.1pt,ropered,line join=round] plot[domain=0:360,samples=181,variable=\t] ({0.36*(sin(\t)+2*sin(2*\t))},{0.36*(cos(\t)-2*cos(2*\t))}) -- cycle;
\draw[line width=5.2pt,ropesand] plot[domain=-0.5:31.6,samples=15,variable=\t] ({0.36*(sin(\t)+2*sin(2*\t))},{0.36*(cos(\t)-2*cos(2*\t))});
\draw[line width=1.1pt,ropered,line cap=round] plot[domain=-3.5:34.6,samples=17,variable=\t] ({0.36*(sin(\t)+2*sin(2*\t))},{0.36*(cos(\t)-2*cos(2*\t))});
\draw[line width=5.2pt,ropesand] plot[domain=88.4:120.5,samples=15,variable=\t] ({0.36*(sin(\t)+2*sin(2*\t))},{0.36*(cos(\t)-2*cos(2*\t))});
\draw[line width=1.1pt,ropered,line cap=round] plot[domain=85.4:123.5,samples=17,variable=\t] ({0.36*(sin(\t)+2*sin(2*\t))},{0.36*(cos(\t)-2*cos(2*\t))});
\draw[line width=5.2pt,ropesand] plot[domain=119.5:151.6,samples=15,variable=\t] ({0.36*(sin(\t)+2*sin(2*\t))},{0.36*(cos(\t)-2*cos(2*\t))});
\draw[line width=1.1pt,ropered,line cap=round] plot[domain=116.5:154.6,samples=17,variable=\t] ({0.36*(sin(\t)+2*sin(2*\t))},{0.36*(cos(\t)-2*cos(2*\t))});
  \node[font=\scriptsize,below] at (0,-1.05) {a tangled \emph{unknotted} rope};
  \draw[->,>=Stealth,line width=0.8pt,ropegrey] (1.45,0.1) -- (2.45,0.1);
  \node[font=\scriptsize,above,ropegrey] at (1.95,0.15) {?};
\end{scope}
% --- far right: round circle
\begin{scope}[shift={(9.35,0.15)}]
  \draw[line width=6pt,ropesand] (0,0) circle (0.72);
  \draw[line width=1.1pt,ropered] (0,0) circle (0.72);
  \node[font=\scriptsize,below] at (0,-1.15) {the round circle};
\end{scope}
\end{tikzpicture}
\caption{Left: thickness is the radius of the largest embedded tube.
Right: the standard thick-unknot question; the middle picture is a
three-crossing diagram of the unknot.  Can every unknotted
unit-thickness configuration be carried to the round configuration by a
fixed-length physical isotopy?  Numerical work motivated a negative answer
\cite{PieranskiGordian}; rigorous Gordian results are known for links
\cite{CowardHass,KusnerKusner,AyalaHass}, while
\cite{AyalaGeometric} proves Gordian unknots in a related thinner model.}
\label{fig:gordian}
\end{figure}

\subsection{Turning the puzzle into topology}

The purpose of this paper is to place questions of this kind inside a single
elementary framework, and to show that once the framework is set up, several
things that look like folklore become theorems, several things that look
like theorems are only conjectures, and the boundary between the two can be
drawn precisely.

The framework is a filtration by geometric budget.  Normalise the rope to
have thickness \(1\).  For a knot type \(K\) and a level \(\Lambda\), put
\begin{equation}\label{eq:Y-intro}
  Y_\Lambda(K)=
  \bigl\{\gamma:\ [\gamma]=K,\ \Thi(\gamma)\ge1,\ \Len(\gamma)\le\Lambda
  \bigr\}\big/\!\sim,
\end{equation}
where \(\sim\) identifies configurations differing by a rigid motion and a
reparametrisation.  As \(\Lambda\) grows, \(Y_\Lambda(K)\) grows; the union
over all \(\Lambda\) is the space of all thick configurations of \(K\).  The
smallest level with \(Y_\Lambda(K)\ne\varnothing\) is the ropelength
\(m_K=\Rop(K)\) of the knot type, and \(I(K)=Y_{m_K}(K)\) is the
\emph{ideal stratum}: the space of tight, or ideal, configurations.

Everything in this paper is a statement about how the topology of
\(Y_\Lambda(K)\) --- or of the parallel smooth model \(\calE_{K,\Lambda}\)
of \cref{sec:filtration} --- changes with \(\Lambda\).  The pictures to keep
in mind are collected in \cref{fig:landscape}: the space of all
configurations of a fixed knot type carries a height function, the sublevels
are the regions below a rising water line, and one wants to know when
islands merge, when loops of configurations become fillable, and when the
picture stabilises.

\begin{figure}[t]
\centering
\begin{tikzpicture}[scale=1.0]
\def\prof{(0,0) .. controls (0.7,1.55) and (1.1,0.35) .. (1.8,0.35)
        .. controls (2.5,0.35) and (2.7,2.05) .. (3.5,2.05)
        .. controls (4.3,2.05) and (4.2,0.75) .. (5.0,0.75)
        .. controls (5.8,0.75) and (6.0,1.85) .. (6.8,1.85)
        .. controls (7.4,1.85) and (7.3,0.72) .. (7.9,0.72)
        .. controls (8.2,0.72) and (8.2,1.15) .. (8.4,1.5)}
% ground
\fill[ropesand] \prof -- (8.4,-0.15) -- (0,-0.15) -- cycle;
% water below the level Lambda_2: the sublevel Y_{Lambda_2}(K)
\begin{scope}
  \clip \prof -- (8.4,2.6) -- (0,2.6) -- cycle;
  \fill[ropelight] (0,-0.2) rectangle (8.4,1.15);
\end{scope}
\draw[line width=1.0pt,ropegrey] \prof;
% water levels
\foreach \y/\l in {0.35/1,1.15/2,1.95/3}{
  \draw[ropeblue,line width=0.9pt,dashed] (-0.25,\y) -- (8.7,\y);
  \node[ropeblue,font=\scriptsize,left] at (-0.25,\y) {$\Lambda_{\l}$};}
% sublevel components at Lambda_2
\draw[ropeblue,line width=1.6pt] (1.03,-0.33) -- (2.72,-0.33);
\draw[ropeblue,line width=1.6pt] (4.28,-0.33) -- (8.3,-0.33);
\node[ropeblue,font=\scriptsize,right] at (8.35,-0.33) {$Y_{\Lambda_2}(K)$};
% minima markers
\filldraw[ropered] (1.8,0.35) circle (1.6pt);
\filldraw[ropered] (5.0,0.75) circle (1.6pt);
\filldraw[ropered] (7.9,0.72) circle (1.6pt);
\node[ropered,font=\scriptsize,above] at (1.8,0.38) {$A$};
\node[ropered,font=\scriptsize,above] at (5.0,0.78) {$B$};
\node[ropered,font=\scriptsize,above] at (7.9,0.75) {$C$};
\node[font=\scriptsize,align=center] at (4.2,-0.72)
  {space of thick configurations of the knot type $K$};
\node[rotate=90,font=\scriptsize] at (-0.95,1.1) {length};
\end{tikzpicture}
\caption{Ropelength persistence in one picture.  The horizontal direction is
the space of all unit-thickness configurations of a fixed knot type; the
height is length.  The sublevel \(Y_\Lambda(K)\) is what lies below the
water line \(\Lambda\); the shaded water and the thick segments show
\(Y_{\Lambda_2}(K)\), which has two components.  Local minima \(A,B,C\) are ideal or critical
configurations; the level at which two of them are first joined below water
is their \emph{connected} merge height, and \cref{thm:merge-ultrametric}
says that this level is attained and gives an ultrametric on connected
components of the ideal stratum.  Path merge, defined by actual paths, has
the same value (\cref{thm:merge-equal}) but is not known to be attained; see
\cref{rem:path-vs-conn}.}
\label{fig:landscape}
\end{figure}

\subsection{Results}

The results are deliberately foundational.  Their purpose is to make the
filtration precise enough that geometric questions, topological questions,
and physical-isotopy questions are not silently conflated.

\begin{enumerate}[label=\textup{(\Alph*)},leftmargin=3.1em]
\item \textbf{Compact capture and exhaustion}
(\cref{thm:compact-family-bound,thm:homotopical-exhaustion}).  Every compact
family of smooth knots has uniformly bounded ropelength.  Consequently the
ropelength filtration is exhaustive for homotopy and homology, and a cofinal
mapping telescope recovers the ordinary knot-space component up to weak
homotopy equivalence.  Thus the Hatcher--Budney homotopy type is the stable
target of the filtration.

\item \textbf{Degree zero: compactness, fixed length, and two merge scales}
(\cref{thm:knot-sublevel-compact,prop:fixed-length-retract,%
thm:merge-ultrametric}).  Every finite \(Y_\Lambda(K)\) is compact
Hausdorff.  Connected merge levels are therefore attained and give an
ultrametric on connected components of the ideal stratum.  On the other
hand \(Y_L(K)\) strongly deformation retracts onto the exact-length shell,
so \emph{path} connectivity in \(Y_L(K)\) is precisely fixed-length
physical isotopy.  The connected merge \(\mu_K\) and the path merge
\(\mu_K^{\rm path}\) satisfy
\(\mu_K=\mu_K^{\rm path}\) (\cref{thm:merge-equal}), the sublevels being
locally path connected after arbitrarily small right relaxation
(\cref{thm:lpc}); only the first infimum is presently known to be attained.
This is the basis of \cref{thm:untying-dichotomy}.
Combining right-relaxed finiteness with pairwise finite merge levels shows
that every fixed sublevel is contained in one path component of some finite
higher sublevel (\cref{cor:uniform-finite-merging}); this applies in
particular to the entire ideal stratum.

\item \textbf{Widths and symmetry}
(\cref{prop:width-properties,thm:orbit-width,thm:vanishing-width}).
Ropelength width gives a non-Archimedean group seminorm after subtraction of
the base level.  Classes carried by ambient isometry orbits have zero excess
width.  For the round \(S^3\) model, Hatcher's orbit descriptions imply
that all \emph{stable} positive-degree classes of torus-knot and hyperbolic-
knot components are represented at the ropelength of a symmetric orbit
basepoint.  This is a statement about stable classes; transient classes
below that level may still contain filtered information.  Satellite models,
with their additional non-isometric parameters, are therefore natural
places to look for positive stable excess width.

\item \textbf{Mirror scales and discrete benchmarks.}
Reflection preserves the filtration, giving a \(\Z/2\)-action for
amphichiral knot types
(\cref{prop:mirror-involution,cor:mirror-barrier-criterion}).  Complete
simple-cubic BFACF computations of the minimal layers give the merge trees
\(4\to1\) for \(4_1\) at levels \(30,32\) and
\(12\to2\to1\) for \(6_3\) at levels \(40,42,44\)
\cite{OzawaDiscrete}.  In the latter case the two level-\(42\) branches are
mirror images and remain disjoint until \(N=44\).  We use these only as finite
benchmarks: no lattice-to-continuum implication is claimed.

\item \textbf{Dynamic refinement by swept area.}
The companion paper \cite{OzawaSwept} equips the regularity-restricted
admissible path problem at each fixed ropelength level with an infimal
swept-area cost.  After endpoint alignment this cost is a genuine extended
metric, not merely a pseudometric; it is finite on each admissible component,
monotone as the ropelength budget increases, and bounded below by an
Euclidean-orbit flat-current distance.  A projected-area-vector calibration
gives an explicit quotient-safe lower bound.  Thus the filtration records
\emph{when} admissible connections appear, while swept area measures a
geometric cost \emph{after} such a connection is available.  The
interpolation paths of \cref{thm:lpc} are Lipschitz, so the level at which
such a cost first becomes finite coincides with the path merge threshold
(\cref{subsec:swept-area}); no convergence of minimal costs is
asserted.

\item \textbf{Finite factorization in a marked weak model}
(\cref{thm:marked-finite-factorization}).  In the \(C^1\) model of
constant-speed, length-one curves with phase, orientation and rotation
retained, every inclusion across a positive ropelength gap factors up to
homotopy through a finite simplicial complex.  Hence the homology
persistence modules of this model are q-tame in every degree.  The same
section records component information carried by \(H_0\), finite terminal
cores, the escape of split link factors, and essential homology persistence
of pole-marked models.  None of these statements is transported to the
symmetry quotients without a further argument.
\end{enumerate}

\subsection{What is standard, what is proved here, and what remains open}
\label{subsec:whats-new}

A recurring source of confusion in this subject is that closely related
models answer different questions.  We therefore record provenance and
scope explicitly.

\begin{table}[p]
\footnotesize
\begin{tabular}{@{}p{5.25cm}p{6.85cm}@{}}
\toprule
\textbf{Statement} & \textbf{Status}\\
\midrule
Existence and \(C^{1,1}\) regularity of ropelength minimisers
 & Standard; Cantarella--Kusner--Sullivan \cite{CKS}, with related
   formulations in \cite{GonzalezMaddocks,SchurichtvdMosel}.\\
Compactness of \(\{\Thi\ge1,\ \Len\le\Lambda\}/\!\sim\)
 (\cref{thm:knot-sublevel-compact})
 & The analytic compactness mechanism is standard in ropelength theory;
   we isolate the quotient statement needed for the filtration.\\
Compact capture and exhaustion
 (\cref{thm:compact-family-bound,thm:homotopical-exhaustion})
 & Elementary consequence of local tubular stability and compactness of
   the parameter space; included because it is the bridge to knot-space
   topology.\\
Exact-length retraction (\cref{prop:fixed-length-retract})
 & Proposition proved here.  It identifies path connectivity of
   \(Y_L(K)\) with fixed-length physical isotopy.\\
Right continuity of connected components and attained connected merge
 (\cref{cor:right-continuity,thm:merge-ultrametric})
 & Consequences of compactness, isolated here in the form needed for
   ropelength filtrations.\\
Path merge and physical Gordian behaviour
 (\cref{rem:path-vs-conn,thm:untying-dichotomy})
 & Path merge is the min--max budget for actual paths.  At a fixed level,
   path connectivity agrees with fixed-length physical isotopy.  Equality with connected merge is proved in \cref{thm:merge-equal}; attainment of the path infimum remains open.\\
Swept-area metric on admissible components
 & Companion result \cite{OzawaSwept}: the infimal swept area of
   regularity-restricted admissible isotopies defines a genuine extended
   metric, finite on each admissible component, with current-theoretic and
   projected-area lower bounds.  Right-relaxed endpoint thresholds agree by \cref{thm:lpc,thm:merge-equal}; convergence of costs is a separate question.\\
Gordian phenomena
 & Rigorous for several link models
   \cite{CowardHass,KusnerKusner,AyalaHass,Bauermeister}.  Ayala
   \cite{AyalaGeometric} proves Gordian unknots in a thinner model
   \(\mathcal U_1\); this does not by itself settle the standard thick model
   used here.\\
Width and isometry-orbit vanishing
 (\cref{prop:width-properties,thm:orbit-width})
 & Formal once compact capture is known.\\
Stable width for torus and hyperbolic knot spaces
 (\cref{thm:vanishing-width})
 & Consequence of Hatcher's orbit models
   \cite{HatcherModuli,HatcherSpaces}; scoped to the round \(S^3\) model and
   an orbit basepoint.\\
Discrete minimal-layer BFACF merge trees for \(4_1\) and \(6_3\)
 & Certified companion computation \cite{OzawaDiscrete}: \(4\to1\) at
   \(30,32\) and \(12\to2\to1\) at \(40,42,44\); discrete and move-system
   specific.\\
Marked weak model: finite factorization, q-tameness
 (\cref{thm:marked-finite-factorization})
 & Proved here; standard nerve device, new input \cref{lem:marked-convex};
   marked \(C^1\) model only.\\
Continuum figure-eight mirror separation
 & Open (\cref{prob:figure-eight-mirror}).\\
Higher Reidemeister--Cerf comparison
 & Programmatic.  Classical one-parameter transversality is supported by
   Queffelec \cite{Queffelec}; the filtered higher comparison remains
   a separate relative-genericity problem in \cite{OzawaFinite}.\\
\bottomrule
\end{tabular}
\caption{Scope and provenance.  The distinction between connected and path
components is deliberately visible: it is mathematically essential in
degree zero.}
\label{tab:status}
\end{table}

The table also suggests a hierarchy of questions.  Degree zero is already
nontrivial in physical link theory, but in our compact knot sublevels the
quantity with the cleanest attainment theorem is the \emph{connected}
merge, not the path merge.  In positive degree, the stable classes of the
best understood torus and hyperbolic components are geometrically cheap in
the round-sphere model, while transient classes remain possible.  Thus the
filtration has two complementary frontiers: understanding path components at
finite budget, and finding genuinely filtered higher-homotopy phenomena.

\subsection{How to read this paper}

\Cref{sec:thickness,sec:knotspaces} are introductory and assume only
multivariable calculus, point-set topology, and the definition of a knot.  A
reader who knows ropelength can skip \cref{sec:thickness}; a reader who
knows spaces of knots can skip \cref{sec:knotspaces}; a reader who knows
both should start at \cref{sec:filtration}.  Statements marked
\emph{Beginner's note} are addressed to a student meeting the subject for
the first time and may be skipped without loss.  \Cref{sec:problems}
contains a graded list of problems, and \cref{app:literature} is an
annotated reading list.

\subsection{Organisation}

\Cref{sec:thickness} develops thickness and ropelength from the definitions,
with the elementary bounds and the numerical landscape.
\Cref{sec:knotspaces} does the same for spaces of knots and states what
Hatcher and Budney prove.  \Cref{sec:filtration} introduces the filtration
and proves compact capture and exhaustion.  \Cref{sec:widths} develops
widths, death levels and the persistence window, and proves the vanishing
theorem.  \Cref{sec:degreezero} is the degree-zero theory: compactness,
right continuity, untying levels, merge ultrametrics, Gordian phenomena,
mirror scales, and the figure-eight mirror problem.  \Cref{sec:bridges} gives
three bridges to the Hatcher--Budney picture, including the pole-marked
complement bundle.  \Cref{sec:program} explains the relation with the
projection-framed diagrammatic models of \cite{OzawaFinite}.
\Cref{sec:marked-finite} proves finite factorization and q-tameness in a
marked weak model and records the accompanying statements on components,
terminal cores, split links, and pole-marked persistence.
\Cref{sec:problems} lists problems by difficulty, and
\cref{app:notation,app:literature} give a notation table and a guide to the
literature.

\section{Thickness and ropelength from scratch}
\label{sec:thickness}

This section is written for a reader who has never used ropelength.  Nothing
in it is new; the aim is to fix conventions carefully, because the
literature contains two of them, and to record the elementary inequalities
that will be used later.

\subsection{Thickness}

Let \(\gamma:S^1\to\R^3\) be an embedded closed curve of class \(C^{1,1}\),
that is, a \(C^1\) curve whose unit tangent is Lipschitz.  For \(r>0\) let
\[
  T_r(\gamma)=\{x\in\R^3:\operatorname{dist}(x,\gamma)\le r\}
\]
be the closed \(r\)-neighbourhood of the image.

\begin{definition}[Thickness]\label{def:thickness}
The \emph{thickness} \(\Thi(\gamma)\) is the supremum of those \(r>0\) for
which \(T_r(\gamma)\) is an embedded solid torus obtained as the union of
the closed normal discs of radius \(r\); equivalently, \(\Thi(\gamma)\) is
the \emph{reach} of the image in the sense of Federer \cite{Federer}.
\end{definition}

Two independent things can destroy a thick tube, and only two.  The tube can
fail because the curve bends too sharply, so that normal discs on the
\emph{same} short arc already overlap; or it can fail because two
\emph{distant} pieces of the curve come close together.  The second is
measured by the \emph{doubly critical self-distance}
\[
  \operatorname{dcsd}(\gamma)
  =\inf\bigl\{\,|\gamma(s)-\gamma(t)|\ :\ s\ne t,\
   \gamma(s)-\gamma(t)\perp\gamma'(s),\gamma'(t)\,\bigr\},
\]
the infimum of lengths of \emph{doubly critical chords}, that is, chords
meeting the curve perpendicularly at both ends.  The following formula is
the standard computational description of thickness.

\begin{theorem}[Thickness formula; Litherland--Simon--Durumeric--Rawdon
\cite{LSDR}, Gonzalez--Maddocks \cite{GonzalezMaddocks}, see also
{\cite[Lemma~1.1]{CKS}}]
\label{thm:thickness-formula}
For an embedded \(C^{1,1}\) curve \(\gamma\),
\[
  \Thi(\gamma)
  =\min\Bigl\{\ \inf_{s}\rho(s),\ \tfrac12\operatorname{dcsd}(\gamma)\Bigr\},
\]
where \(\rho(s)=1/|\kappa(s)|\) is the radius of curvature, defined almost
everywhere.  Equivalently \textup{(}Gonzalez--Maddocks
\cite{GonzalezMaddocks}, {\cite[Lemma~1.1]{CKS}}\textup{)},
\begin{equation}\label{eq:three-point}
  \Thi(\gamma)=\inf\bigl\{\,R(x,y,z)\ :\ x,y,z\ \text{distinct points of}\
  \gamma\,\bigr\},
\end{equation}
where \(R(x,y,z)\) is the radius of the circle through \(x,y,z\).
\end{theorem}

\begin{remark}[Beginner's note: why \eqref{eq:three-point} is the useful form]
\label{rem:three-point}
Formula \eqref{eq:three-point} expresses thickness as an infimum of
quantities each of which is a continuous function of three points on the
curve.  Consequently the condition \(\Thi\ge1\) is \emph{closed} under
uniform convergence: if \(\gamma_n\to\gamma\) uniformly and every
\(\gamma_n\) has thickness at least \(1\), then so does \(\gamma\).  This
one observation is what makes the compactness theorem of
\cref{sec:degreezero} work, and it is the reason the three-point form,
rather than \cref{thm:thickness-formula} itself, is the workhorse of the
subject.
\end{remark}

\begin{figure}[t]
\centering
\begin{tikzpicture}[scale=1.0]
% (a) curvature obstruction: bend radius rho smaller than the tube radius r
\begin{scope}
  \draw[line width=26pt,ropelight,line cap=butt,line join=round]
     (-1.7,-0.3) -- (0.2,-0.3) arc (-90:90:0.3) -- (-1.7,0.3);
  % points within distance r of two different parts of the curve
  \fill[ropeblue!30] (-1.7,-0.16) rectangle (0.2,0.16);
  \fill[ropeblue!30] (0.2,0) circle (0.16);
  \draw[line width=1.1pt,ropeblue]
     (-1.7,-0.3) -- (0.2,-0.3) arc (-90:90:0.3) -- (-1.7,0.3);
  \draw[ropered,line width=0.6pt,dashed] (0.2,0) circle (0.3);
  \filldraw[ropered] (0.2,0) circle (1pt);
  \draw[->,>=Stealth,ropered,line width=0.6pt] (0.2,0) -- ++(30:0.3);
  \node[ropered,font=\scriptsize] at (0.74,0.12) {$\rho$};
  \draw[<->,>=Stealth,ropegrey,line width=0.5pt] (-1.35,-0.3) -- (-1.35,-0.76);
  \node[ropegrey,font=\scriptsize,left] at (-1.37,-0.55) {$r$};
  \node[font=\small] at (-0.5,-1.3) {(a) too sharp a bend};
\end{scope}
% (b) self-contact obstruction
\begin{scope}[shift={(5.2,0)}]
  \def\up{(-1.7,0.95) .. controls (-0.6,0.0) and (0.6,0.0) .. (1.7,0.95)}
  \def\dn{(-1.7,-0.95) .. controls (-0.6,0.0) and (0.6,0.0) .. (1.7,-0.95)}
  \begin{scope}[opacity=0.45]
    \draw[line width=26pt,ropeblue!45,line cap=butt] \up;
    \draw[line width=26pt,ropeblue!45,line cap=butt] \dn;
  \end{scope}
  \draw[line width=1.1pt,ropeblue] \up;
  \draw[line width=1.1pt,ropeblue] \dn;
  \draw[<->,>=Stealth,ropered,line width=0.8pt] (0,0.24) -- (0,-0.24);
  \node[ropered,font=\scriptsize,right] at (0.08,0) {$\operatorname{dcsd}$};
  \node[font=\small] at (0,-1.3) {(b) two strands too close};
\end{scope}
\end{tikzpicture}
\caption{The only two obstructions to a thick tube, and the two terms of
\cref{thm:thickness-formula}.  In (a) the radius of curvature \(\rho\) is
smaller than the tube radius \(r\); in (b) a doubly critical chord is shorter
than \(2r\).  In both cases the tube of radius \(r\) overlaps itself (darker
shading).  Thickness is the smaller of
\(\inf\rho\) and \(\tfrac12\operatorname{dcsd}\).}
\label{fig:two-obstructions}
\end{figure}

\begin{remark}[Beginner's note]
It is tempting to define thickness as ``half the minimal distance between
non-adjacent points''.  That does not work: nearby points of the curve are
always close, so the infimum over \emph{all} pairs is \(0\).  The role of
the perpendicularity condition in \(\operatorname{dcsd}\) is exactly to
discard such pairs, keeping only true near-contacts of distinct strands.
The curvature term is then needed separately, since a very tight bend has no
short doubly critical chord at all.
\end{remark}

\subsection{Ropelength and its two normalisations}

\begin{definition}[Ropelength]\label{def:ropelength}
\(\displaystyle \Rop(\gamma)=\frac{\Len(\gamma)}{\Thi(\gamma)}\), and for a
knot type \(K\),
\(\displaystyle m_K=\Rop(K)=\inf\{\Rop(\gamma):[\gamma]=K\}\).
\end{definition}

Ropelength is invariant under rigid motions, reparametrisation and
scaling, since scaling by \(s>0\) multiplies both \(\Len\) and \(\Thi\) by
\(s\) (\cref{fig:scale}).  Consequently one may always normalise, and we
shall normalise in the way that is most convenient for filtration
arguments: \emph{thickness at least \(1\)}, so that ropelength becomes
length.

\begin{convention}[Normalisation]\label{conv:normalisation}
We take \(\Thi\) to be the \emph{radius} of the largest embedded tube, as in
\cite{CKS}.  With this convention the round circle of radius \(1\) has
\(\Thi=1\), \(\Len=2\pi\) and \(\Rop=2\pi\).  Part of the literature, for
instance \cite{DDS}, uses the tube \emph{diameter} instead, which halves all
ropelength values.  Numerical values quoted below have been converted to the
radius convention; the reader comparing sources should always check which
convention is in force, since a factor of \(2\) is easy to lose.
\end{convention}

\begin{figure}[b]
\centering
\begin{tikzpicture}[scale=1.0]
  \fill[ropesand] (0,0) circle (0.9);
  \draw[line width=1pt,ropered] (0,0) circle (0.45);
  \filldraw[ropegrey] (0,0) circle (0.8pt);
  \draw[->,>=Stealth,ropegrey,line width=0.5pt] (0,0) -- (0.45,0);
  \node[font=\tiny,ropegrey,above] at (0.23,0) {$1$};
  \node[font=\scriptsize,below] at (0,-1.0) {$\Len=2\pi,\ \Thi=1$};
  \node[font=\small] at (2.1,0) {$\Rop=2\pi$};
  \fill[ropesand] (5.0,0) circle (1.8);
  \draw[line width=1pt,ropered] (5.0,0) circle (0.9);
  \filldraw[ropegrey] (5.0,0) circle (0.8pt);
  \draw[->,>=Stealth,ropegrey,line width=0.5pt] (5.0,0) -- (5.9,0);
  \node[font=\tiny,ropegrey,above] at (5.45,0) {$2$};
  \node[font=\scriptsize,below] at (5.0,-1.9) {$\Len=4\pi,\ \Thi=2$};
  \node[font=\small] at (7.9,0) {$\Rop=2\pi$};
\end{tikzpicture}
\caption{Ropelength is scale invariant, so it is a property of the shape and
not of the size.  The tube of a round circle is as large as possible: it
reaches the centre.  This is why the normalisation \(\Thi\ge1\) of
\eqref{eq:Y-intro} loses nothing.}
\label{fig:scale}
\end{figure}

\subsection{Two elementary bounds a beginner can prove}

The following two propositions are the whole of what one needs to get a feel
for the numbers.  Both proofs are three lines, and both are worth doing once.

\begin{proposition}[The round circle is the tight unknot]
\label{prop:unknot-2pi}
If \(\gamma\) is a closed \(C^{1,1}\) curve with \(\Thi(\gamma)\ge1\), then
\(\Len(\gamma)\ge2\pi\), with equality if and only if \(\gamma\) is a round
circle of radius \(1\).  In particular \(\Rop(U)=2\pi\) for the unknot
\(U\), and the ideal stratum \(I(U)\) is a single point.
\end{proposition}

\begin{proof}
By \cref{thm:thickness-formula}, \(\Thi(\gamma)\ge1\) forces
\(\rho\ge1\), that is \(|\kappa|\le1\) almost everywhere.  Hence
\[
  \Len(\gamma)=\int_{\gamma}1\,ds\ \ge\ \int_\gamma|\kappa|\,ds\ \ge\ 2\pi ,
\]
the last step being Fenchel's theorem \cite{Fenchel}: a closed curve in
\(\R^3\) has total curvature at least \(2\pi\).  Equality forces
\(|\kappa|\equiv1\) almost everywhere and total curvature exactly \(2\pi\);
the equality case of Fenchel's theorem says the curve is a plane convex
curve, and a plane convex curve of constant curvature \(1\) is a circle of
radius \(1\).  Conversely the unit circle has thickness \(1\) and length
\(2\pi\).  Since all unit circles are congruent, \(I(U)\) is one point.
\end{proof}

\begin{proposition}[A knotted rope needs \(4\pi\)]
\label{prop:farymilnor-bound}
If \(\gamma\) is a closed \(C^{1,1}\) curve with \(\Thi(\gamma)\ge1\) whose
knot type is nontrivial, then \(\Len(\gamma)>4\pi\approx12.566\).
Consequently \(m_K>4\pi\) for every nontrivial knot type \(K\).
\end{proposition}

\begin{proof}
As above \(|\kappa|\le1\) almost everywhere, so
\(\Len(\gamma)\ge\int_\gamma|\kappa|\,ds\), and the F\'ary--Milnor theorem
\cite{Fary,Milnor} says that a nontrivially knotted closed curve has total
curvature greater than \(4\pi\).
\end{proof}

\begin{remark}[Beginner's note]
\Cref{prop:farymilnor-bound} is the prototype of every lower bound in the
subject: convert a \emph{topological} statement (F\'ary--Milnor) into a
\emph{metric} one using the pointwise inequality \(|\kappa|\le1/\Thi\).  The
sharper bounds below replace F\'ary--Milnor by more refined geometric
inputs, but the mechanism is the same.
\end{remark}

\subsection{The numerical landscape}

\Cref{fig:numberline} places several useful benchmark bounds next to two
widely used numerical estimates.  Only the unknot value \(2\pi\) is an exact
knot-type ropelength in this figure; the dashed nontrivial-knot marks are
rigorous lower bounds, while the trefoil and figure-eight marks are
numerical.

\begin{figure}[t]
\centering
\begin{tikzpicture}[x=0.248cm,y=1cm]
\draw[->,>=Stealth,line width=0.9pt] (0,0) -- (46.5,0);
\foreach \x in {0,5,10,15,20,25,30,35,40,45}
  \draw[line width=0.7pt] (\x,0) -- (\x,-0.1) node[below,font=\scriptsize] {\x};
\node[font=\scriptsize,right] at (46.7,0) {$\Rop$};
% lower bounds, above the axis, staggered
\draw[ropegreen,line width=1.2pt] (6.283,0) -- (6.283,1.42);
\node[ropegreen,font=\scriptsize,above] at (6.283,1.44) {$2\pi$: unknot};
\draw[ropeblue,line width=1pt,dashed] (12.566,0) -- (12.566,0.72);
\node[ropeblue,font=\scriptsize,above] at (12.566,0.74)
  {$4\pi$: F\'ary--Milnor};
\draw[ropeblue,line width=1pt,dashed] (21.45,0) -- (21.45,1.42);
\node[ropeblue,font=\scriptsize,above] at (21.45,1.44) {$21.45$: \cite{CKS}};
\draw[ropeblue,line width=1pt,dashed] (31.32,0) -- (31.32,0.72);
\node[ropeblue,font=\scriptsize,above] at (31.32,0.74) {$31.32$: \cite{DDS}};
% computed values, below the axis
\draw[ropered,line width=1.2pt] (32.74,0) -- (32.74,-0.66);
\node[ropered,font=\scriptsize,below] at (31.6,-0.68) {$32.74$: $3_1$};
\draw[ropered,line width=1.2pt] (42.09,0) -- (42.09,-1.18);
\node[ropered,font=\scriptsize,below] at (41.2,-1.2) {$42.09$: $4_1$};
\end{tikzpicture}
\caption{The ropelength number line, in the radius convention of
\cref{conv:normalisation}.  Dashed marks are proved lower bounds valid for
\emph{every} nontrivial knot; solid marks are values of specific knot types.
The gap between \(31.32\) and \(32.74\) is about \(4.5\%\): the general
lower bound is already very close to the trefoil's value.  Numerical values
are from \cite{ACPR}; the tabulation used here is the one reproduced in
\cite{Ropelength12}.}
\label{fig:numberline}
\end{figure}

\begin{remark}
The reader should treat the two solid marks in \cref{fig:numberline} with
the appropriate caution.  Values such as \(32.74\) and \(42.09\) are outputs
of constrained gradient descent \cite{ACPR}; they are convincing but are not
proofs, and no nontrivial knot type has a ropelength minimiser that is known
in closed form.  This is one reason why the subject is difficult, and also
one reason why the \emph{qualitative} structural questions studied in this
paper are worth asking: they can be settled without knowing any minimiser
explicitly.
\end{remark}

\subsection{Minimisers exist, and are only \texorpdfstring{$C^{1,1}$}{C11}}

\begin{theorem}[Cantarella--Kusner--Sullivan {\cite[Theorem 2.9]{CKS}}]
\label{thm:cks-existence}
In every knot or link type there is a ropelength minimiser, and every
minimiser is a \(C^{1,1}\) curve.  Minimisers need not be \(C^2\).
\end{theorem}

Two consequences shape the whole theory.  First, the natural regularity
class is \(C^{1,1}\) and not \(C^\infty\); tight configurations have arcs of
curvature exactly \(1\) meeting straight segments, and the curvature jumps.
Second, there is no general uniqueness theorem for minimisers up to
congruence.  The minimum level set is therefore naturally treated as a
space, which we call the ideal stratum, and its topology is the subject of
\cref{sec:degreezero}.  Criticality theory for
ropelength, that is the first-order theory of this variational problem, was
developed in \cite{CFKS}; a symmetric criticality principle, showing that a
symmetric knot type has critical configurations with the same symmetry, is
proved in \cite{CEFM}, together with examples of knot types possessing
several non-congruent critical configurations.

\begin{remark}[Why several critical configurations matter here]
\label{rem:several-critical}
If ropelength had a unique critical point in each knot type, the geometry of
\cref{fig:landscape} would be a single basin and there would be nothing to
say.  The examples of \cite{CEFM} show that the landscape really does have
several critical points; whether distinct critical points can lie in
\emph{different connected components} of a sublevel is precisely the
Gordian-type question of \cref{sec:degreezero}.
\end{remark}

\subsection{The change of viewpoint}

Everything after this point rests on one reinterpretation, which we state
plainly because it is the conceptual content of the paper.

\begin{quote}
\emph{Fix the rope: work in the space of all configurations of thickness at
least one.  On this space, length is a function.  Then ropelength theory is
the study of the sublevel sets of that function, and classical ropelength
minimisation is only the study of its minimum.}
\end{quote}

The minimum \(m_K\) is one number; the sublevel sets form a filtration.  The
questions the filtration asks --- when do two basins merge, which families
of configurations fit under a given budget, when does the picture stabilise
--- are invisible to the minimum alone.

\section{Spaces of knots from scratch}
\label{sec:knotspaces}

This section fixes the topological models and states, without proof, what is
known about them.  A reader who wants more should consult
\cref{app:literature}.

\subsection{Four models}

Let \(M\) be \(\R^3\) or \(S^3\) with a fixed Riemannian metric.

\begin{enumerate}[label=\textup{(\arabic*)},leftmargin=2.4em]
\item \emph{Parameterised knots.} \(\Emb^\infty(S^1,M)\), the space of
smooth embeddings with the \(C^\infty\) topology.  Its set of path
components is the set of oriented knot types; a path is an isotopy.  Write
\(\calE_K(M)\) for the component of \(K\).
\item \emph{Unparameterised knots.}
\(\calU_K(M)=\calE_K(M)/\Diff^+(S^1)\), or
\(\mathscr K_K(M)=\Emb^\infty_K(S^1,M)/\Diff(S^1)\) in the unoriented case.
Since \(\Diff^+(S^1)\simeq S^1\), these differ from (1) by a circle's worth
of information, invisible to \(\pi_0\) but not to \(\pi_1\).
\item \emph{Long knots.} \(\calK\) is the space of smooth embeddings
\(f:\R\hookrightarrow\R^3\) agreeing with a fixed line outside a compact
interval.  Its component of \(f\) is written \(\calK_f\).
\item \emph{Shape space.}
\(Y_\infty(K)=\{\gamma:[\gamma]=K,\ \Thi(\gamma)\ge1\}/\!\sim\) as in
\eqref{eq:Y-intro}, with rigid motions and reparametrisations divided out.
This is the model in which length is a function and
\cref{sec:degreezero} takes place.
\end{enumerate}

\begin{figure}[t]
\centering
\begin{tikzpicture}[scale=1.0]
\begin{scope}
\draw[rounded corners=22pt,fill=ropesand,draw=ropegrey,line width=0.7pt]
  (-0.2,-0.2) rectangle (3.2,1.7);
\draw[rounded corners=18pt,fill=ropesand,draw=ropegrey,line width=0.7pt]
  (3.75,-0.05) rectangle (6.3,1.55);
\draw[rounded corners=14pt,fill=ropesand,draw=ropegrey,line width=0.7pt]
  (6.85,0.2) rectangle (8.7,1.3);
\node[font=\scriptsize] at (1.5,1.95) {$\calE_U$ (unknot)};
\node[font=\scriptsize] at (5.02,1.8) {$\calE_{3_1}$ (trefoil)};
\node[font=\scriptsize] at (7.8,1.55) {$\calE_{4_1}$};
\node[font=\scriptsize,ropegrey] at (10.0,0.75) {$\cdots$};
% a path inside one component
\draw[ropered,line width=1pt,->,>=Stealth]
  (4.15,0.55) .. controls (4.7,1.3) and (5.4,0.45) .. (5.95,1.05);
\filldraw[ropered] (4.15,0.55) circle (1.5pt);
\filldraw[ropered] (5.95,1.05) circle (1.5pt);
\node[ropered,font=\scriptsize] at (5.15,0.15) {isotopy};
% a loop
\draw[ropeblue,line width=1pt] (1.5,0.75) circle (0.62);
\node[ropeblue,font=\scriptsize] at (1.5,0.75) {loop};
\end{scope}
\end{tikzpicture}
\caption{The space of knots.  Classical knot theory computes the set of
components.  The subject of this paper is what happens \emph{inside} one
component: paths are isotopies, loops are cyclic motions of a knot, and
higher spheres are multi-parameter families.}
\label{fig:knotspace-cartoon}
\end{figure}

\subsection{The first interesting loop}

Let \(f\) be a long knot, agreeing with the \(x\)-axis outside a compact
interval, and let \(R_t\in SO(3)\) be the rotation by angle \(t\) about the
\(x\)-axis.  Then \(t\mapsto R_t\circ f\), \(t\in[0,2\pi]\), is a loop in
\(\calK_f\) based at \(f\): the knot is spun once about its own axis and
returns to its starting position.

\begin{figure}[t]
\centering
\begin{tikzpicture}[scale=0.78]
\foreach \i/\c/\lab in {0/1/0,1/0.5/60,2/-0.5/120,3/-1/180}{
  \begin{scope}[shift={(\i*3.55,0)}]
    \draw[ropegrey,line width=0.6pt,dashed] (-1.6,0) -- (1.6,0);
    \begin{scope}[yscale=\c]
      \draw[line width=5pt,ropesand,line cap=round]
        (-1.45,0) .. controls (-0.65,0) and (-0.55,0.8) .. (0,0.8)
                  .. controls (0.55,0.8) and (0.5,-0.45) .. (-0.2,-0.45)
                  .. controls (-0.85,-0.45) and (-0.5,0.45) .. (0.35,0.32)
                  .. controls (0.9,0.24) and (0.95,0) .. (1.45,0);
      \draw[line width=1.1pt,ropeblue,line cap=round]
        (-1.45,0) .. controls (-0.65,0) and (-0.55,0.8) .. (0,0.8)
                  .. controls (0.55,0.8) and (0.5,-0.45) .. (-0.2,-0.45)
                  .. controls (-0.85,-0.45) and (-0.5,0.45) .. (0.35,0.32)
                  .. controls (0.9,0.24) and (0.95,0) .. (1.45,0);
    \end{scope}
    \node[font=\scriptsize,below] at (0,-1.15) {$t=\lab^\circ$};
  \end{scope}}
\foreach \i in {0,1,2}{
  \draw[->,>=Stealth,ropered,line width=0.8pt]
    ({\i*3.55+1.75},0.15) arc (200:340:0.5);}
\end{tikzpicture}
\caption{The Gramain loop: spin a long knot once about its axis.  Every
member of this family is a rigid rotation of the first, so the whole loop
lies in a single level set of ropelength.  Gramain \cite{Gramain} proved
that for a nontrivial knot the loop is not null-homotopic; by
\cref{cor:gramain-width} it nevertheless costs no rope at all.}
\label{fig:gramain}
\end{figure}

\begin{theorem}[Gramain \cite{Gramain}]\label{thm:gramain}
For a nontrivial long knot \(f\), the class \([G_f]\in\pi_1(\calK_f,f)\) of
the rotation loop is nontrivial, and it remains nonzero in
\(H_1(\calK_f)\).
\end{theorem}

\subsection{What the terminal spaces look like}

The homotopy type of a knot-space component in dimension \(3\) is known in
many cases; this is a body of work resting on Hatcher's proof of the Smale
conjecture and on the geometrisation of \(3\)-manifolds.  We record what we
use.

\begin{theorem}[Hatcher \cite{HatcherSmale}]\label{thm:unknot-space}
The space of unknotted smooth circles in \(S^3\) deformation retracts onto
the space of great circles, which is the Grassmannian
\(\mathrm{Gr}_2(\R^4)\) of \(2\)-planes through the origin in \(\R^4\).
\end{theorem}

\begin{theorem}[Hatcher {\cite{HatcherSpaces,HatcherModuli}}]
\label{thm:hatcher-models}
Let \(K\subset S^3\) be a knot and \(\mathscr K_K(S^3)\) its knot-space
component.
\begin{enumerate}[label=\textup{(\roman*)},leftmargin=2.6em]
\item If \(K\) is a nontrivial torus knot, \(\mathscr K_K(S^3)\) has the
homotopy type of \(SO(4)/G_K\) for a subgroup \(G_K\) containing a circle,
namely the stabiliser of a standard placement of \(K\) on a Clifford torus.
\item If \(K\) is hyperbolic, \(\mathscr K_K(S^3)\) has the homotopy type of
\(SO(4)/\Gamma_K\), where \(\Gamma_K\) is the finite group of
orientation-preserving isometries of the complete hyperbolic structure on
\(S^3\setminus K\).
\item For satellite knots, Hatcher's model has the form
\((SO(4)\times X_K)/\Gamma_K\) with a homotopy equivalence to
\(\mathscr K_K(S^3)\).  The auxiliary space \(X_K\) records the non-rigid
satellite parameters: in nested cases it contains torus factors, while in
general the construction also involves configuration space factors.  Its
precise form is determined by the satellite/JSJ structure.
\end{enumerate}
\end{theorem}

\begin{remark}\label{rem:linearisation}
Hatcher's statements in the hyperbolic and satellite cases were formulated
modulo a linearisation hypothesis for finite group actions on \(S^3\); that
hypothesis is now a theorem, by the equivariant geometrisation results of
Dinkelbach and Leeb \cite{DinkelbachLeeb}.  The essential point for us is
structural and is stated separately as \cref{hyp:orbit-retract}: in cases
(i) and (ii) the whole knot space retracts onto a \emph{single orbit} of the
isometry group of \(S^3\), whereas in case (iii) there are extra \(X_K\)-directions that are not
ambient isometry directions.
\end{remark}

\begin{theorem}[Hatcher, Budney \cite{HatcherSpaces,BudneyTopology}]
\label{thm:budney-long}
Every component of the space \(\calK\) of long knots is aspherical, and
\[
  \calK_f\simeq
  \begin{cases}
   \ast,&f\ \text{unknotted},\\
   S^1,&f\ \text{a torus knot},\\
   S^1\times S^1,&f\ \text{hyperbolic}.
  \end{cases}
\]
For general \(f\) the component is described recursively by fibre bundles
built from the JSJ decomposition of the complement, and \(\calK\) is a free
little \(2\)-cubes object on the space of long prime knots
\cite{BudneyLittleCubes}.
\end{theorem}

\begin{table}[htbp]
\small
\begin{tabular}{@{}lll@{}}
\toprule
knot type & \(\mathscr K_K(S^3)\) & \(\calK_f\) (long)\\
\midrule
unknot & \(\simeq\mathrm{Gr}_2(\R^4)\) & \(\simeq\ast\)\\
torus knot & \(\simeq SO(4)/G_K\), \(\dim=5\) & \(\simeq S^1\)\\
hyperbolic knot & \(\simeq SO(4)/\Gamma_K\), \(\Gamma_K\) finite &
  \(\simeq S^1\times S^1\)\\
satellite knot & \(\simeq (SO(4)\times X_K)/\Gamma_K\), \(X_K\) from JSJ data &
  bundle-theoretic \\
\bottomrule
\end{tabular}
\caption{The terminal homotopy types, from
\cref{thm:unknot-space,thm:hatcher-models,thm:budney-long}.  These are the
\(\Lambda\to\infty\) limits of everything computed in this paper.}
\label{tab:terminal}
\end{table}

\subsection{Why filter?}

\Cref{tab:terminal} is a complete answer to one question and no answer at
all to another.  It says exactly which motions of a knot exist; it says
nothing about which motions are available to a rope of positive thickness
carrying a fixed amount of slack.  A hyperbolic knot has a torus worth of
motions, but a hyperbolic knot tied in a rope only barely long enough to
carry it may have no room to perform them.

Filtering by ropelength is the minimal way to ask the second question while
keeping the first one as a limit.  Everything below is organised by the
following slogan, which \cref{sec:filtration} makes into a theorem:
\emph{finite levels see geometry, the limit sees topology, and nothing is
lost in passing to the limit.}

\section{The ropelength filtration and compact capture}
\label{sec:filtration}

\subsection{Sublevels}

\begin{definition}[Ropelength sublevels]\label{def:sublevels}
For a knot type \(K\) in \(M\) and \(\Lambda>0\), set
\[
  \calE_{K,\Lambda}(M)=\{\gamma\in\calE_K(M):\Rop(\gamma)\le\Lambda\},
\]
\[
  \calE_{K,<\Lambda}(M)=\{\gamma:\Rop(\gamma)<\Lambda\}.
\]
The images in \(\calU_K(M)\) and \(\mathscr K_K(M)\) are denoted
\(\calU_{K,\Lambda}(M)\) and \(\mathscr K_{K,\Lambda}(M)\).  We write
\(\wh\calE_{K,\Lambda}(M)\) for the corresponding space of \(C^{1,1}\)
embeddings of positive thickness, topologised by the norm
\(\|u\|_{C^{1,1}}=\|u\|_{C^0}+\|\nabla u\|_{C^0}+\operatorname{Lip}(\nabla
u)\) in normal graph charts.
\end{definition}

Here, for a general Riemannian \((M,g)\), the thickness \(\Thi_g(\gamma)\)
is the normal injectivity radius of \(\gamma\), that is the supremum of the
radii for which the normal exponential map is an embedding, and
\(\Rop_g(\gamma)=\Len_g(\gamma)/\Thi_g(\gamma)\); for \(M=\R^3\) this is
\cref{def:thickness,def:ropelength}.  We write \(\Rop_{S^3}\) when \(M=S^3\)
carries the round metric of radius one, and drop the subscript when the
metric is clear.

As sets, \(\calE_K(M)=\bigcup_{\Lambda>0}\calE_{K,\Lambda}(M)\), simply
because every smooth embedding has positive thickness and finite length.
The content of this section is that the union is exhaustive in a much
stronger sense.

\begin{remark}[Which model, and why it matters]
\label{rem:model-choice}
The distinction between \(\calE_K\), \(\calU_K\) and a quotient by ambient
isometries must be specified in every degree. Connected symmetry groups preserve path-component sets; disconnected symmetries may identify components.  If a
group \(G\) acts on \(M\) by isometries, there are three constructions:
\[
  \calU_{K,\Lambda}(M),\qquad
  \calU_{K,\Lambda}(M)/G,\qquad
  EG\times_G\calU_{K,\Lambda}(M).
\]
The first keeps ambient motion, the second discards it, the third keeps it
in the form of isotropy data.  Since the symmetric placements in
\cref{thm:hatcher-models} have large stabilisers, the naive quotient
destroys exactly what one wants to see, and we use the Borel construction
whenever equivariance is relevant.
\end{remark}

\subsection{Local tubular stability}

The only geometric input in this section is the following lemma, and the
only tool in its proof is the tubular neighbourhood theorem.

\begin{lemma}[Local tubular stability]\label{lem:local-tubular-stability}
Let \(\gamma\in\Emb^\infty(S^1,M)\).  There exist a neighbourhood
\(V_\gamma\) of \(\gamma\) in \(\Emb^\infty(S^1,M)\) and constants
\(r_\gamma>0\), \(L_\gamma<\infty\) with
\[
  \Thi(\eta)\ge r_\gamma
  \quad\text{and}\quad
  \Len(\eta)\le L_\gamma
  \qquad\text{for all }\eta\in V_\gamma .
\]
In particular \(\Rop\) is locally bounded on \(\Emb^\infty(S^1,M)\).
\end{lemma}

\begin{proof}
By the tubular neighbourhood theorem choose \(r>0\) so that the normal
exponential map of \(\gamma\) is an embedding on the closed disc bundle of
radius \(2r\).  Let \(N_{2r}\) denote the resulting embedded tube.

\emph{Step 1: nearby curves are normal graphs.}  There is a \(C^2\)
neighbourhood of \(\gamma\) in which every embedding \(\eta\) has image
inside \(N_{r/2}\) and meets each normal disc of \(N_{2r}\) transversally in
exactly one point, so that \(\eta\) is the image of a normal section
\(u_\eta\) with \(\|u_\eta\|_{C^2}\) as small as we please.  This is the
standard local normal graph description.

\emph{Step 2: curvature.}  Since \(\|u_\eta\|_{C^2}\) is small, the
curvature of \(\eta\) is uniformly close to that of \(\gamma\); shrinking
the neighbourhood we may assume \(\inf\rho_\eta\ge r\).

\emph{Step 3: no distant self-contact.}  Suppose no neighbourhood as claimed
existed.  Then there are \(\eta_n\to\gamma\) in \(C^2\) and parameters
\(s_n\ne t_n\) realising doubly critical chords of \(\eta_n\) of length
\(<2r\) tending to a value \(<2r\).  Passing to a subsequence,
\(s_n\to s\), \(t_n\to t\).  If \(s\ne t\), continuity gives a doubly
critical chord of \(\gamma\) of length \(\le\liminf|\eta_n(s_n)-\eta_n(t_n)|<2r\),
contradicting \(\Thi(\gamma)\ge2r\).  If \(s=t\), then for large \(n\) the
two points lie on a short subarc of \(\eta_n\) which, by Step 1, is a graph
over a short subarc of \(\gamma\) with small \(C^2\) norm; on such an arc
the chord \(\eta_n(s_n)-\eta_n(t_n)\) makes an angle with
\(\eta_n'(s_n)\) tending to \(\pi/2\) only if the arc length between the two
parameters is bounded below, which contradicts \(s_n-t_n\to0\).  Hence
\(\tfrac12\operatorname{dcsd}(\eta)\ge r\) on a possibly smaller
neighbourhood.

By \cref{thm:thickness-formula}, Steps 2 and 3 give \(\Thi(\eta)\ge r\).
Finally length is continuous in the \(C^1\) topology, so it is bounded on a
possibly smaller neighbourhood.  Take \(r_\gamma=r\).
\end{proof}

\begin{proposition}[Positive-reach version]\label{prop:c11-local-bound}
The conclusion of \cref{lem:local-tubular-stability} holds for
\(\wh{\calE}_K(M)\) with the strong \(C^{1,1}\) topology.
\end{proposition}

\begin{proof}
In a strong \(C^{1,1}\) normal graph chart the speeds are bounded away from
\(0\) and the Lipschitz constants of the unit tangents are uniformly
bounded, so the curvature term of \cref{thm:thickness-formula} is uniformly
positive.  The argument of Step 3 above uses only \(C^1\) convergence
together with a uniform Lipschitz bound on tangents and therefore applies
verbatim.
\end{proof}

\subsection{Compact capture}

\begin{theorem}[Compact capture]\label{thm:compact-family-bound}
Let \(P\) be a compact topological space and \(F:P\to\calE_K(M)\)
continuous.  Then \(\sup_{p\in P}\Rop(F(p))<\infty\); that is,
\(F(P)\subset\calE_{K,\Lambda}(M)\) for some finite \(\Lambda\).  The same
holds in \(\wh{\calE}_K(M)\).
\end{theorem}

\begin{proof}
\(F(P)\) is compact.  Cover it by the neighbourhoods \(V_\gamma\) of
\cref{lem:local-tubular-stability} and extract a finite subcover
\(V_{\gamma_1},\dots,V_{\gamma_N}\).  With
\(r=\min_j r_{\gamma_j}>0\) and \(L=\max_j L_{\gamma_j}<\infty\) every
member of the family has thickness \(\ge r\) and length \(\le L\), hence
ropelength \(\le L/r\).  For the \(C^{1,1}\) case use
\cref{prop:c11-local-bound}.
\end{proof}

\begin{figure}[t]
\centering
\begin{tikzpicture}[scale=1.0]
% parameter space
\draw[line width=0.9pt,ropeblue,fill=ropelight] (-0.1,0.55) circle (0.72);
\node[font=\scriptsize,below] at (-0.1,-0.35) {$P$ compact};
\draw[->,>=Stealth,line width=0.9pt] (0.85,0.55) -- (1.95,0.55);
\node[font=\scriptsize,above] at (1.4,0.6) {$F$};
% knot space with cover
\begin{scope}[shift={(2.3,0)}]
  \draw[rounded corners=20pt,fill=ropesand,draw=ropegrey,line width=0.7pt]
    (0,-0.7) rectangle (6.6,1.85);
  \node[font=\scriptsize,below right] at (0.1,-0.75) {$\calE_K(M)$};
  % image
  \draw[ropeblue,line width=1.1pt,fill=ropelight,fill opacity=0.6]
    (1.5,0.6) .. controls (2.1,1.45) and (3.6,1.5) .. (4.3,0.85)
              .. controls (4.9,0.3) and (3.2,-0.25) .. (2.2,0.1)
              .. controls (1.6,0.3) and (1.3,0.35) .. (1.5,0.6);
  % cover discs
  \foreach \x/\y in {1.75/0.5,2.65/0.95,3.6/0.95,4.25/0.55,3.05/0.3}{
    \draw[ropered,line width=0.55pt,dashed] (\x,\y) circle (0.6);}
  \node[ropeblue,font=\scriptsize,fill=ropelight,inner sep=1pt] at (3.0,0.62) {$F(P)$};
  \node[ropered,font=\scriptsize,right] at (5.0,1.5) {$V_{\gamma_j}$};
\end{scope}
\end{tikzpicture}
\caption{Compact capture (\cref{thm:compact-family-bound}).  On each
\(V_{\gamma_j}\) thickness is bounded below and length above; finitely many
suffice to cover the compact image, so the whole family lives under one
finite budget.  Every exhaustion statement in this paper is this picture.}
\label{fig:compact-capture}
\end{figure}

\begin{lemma}[Descent to open quotients]\label{lem:quotient-descent}
Let a topological group \(G\) act on \(X\) with open orbit map \(q:X\to
X/G\), and let \(h:X\to[0,\infty)\) be \(G\)-invariant and locally bounded.
Then the induced \(\bar h\) on \(X/G\) is locally bounded, so every compact
subset of \(X/G\) lies in one sublevel of \(\bar h\).
\end{lemma}

\begin{proof}
If \(h\le C\) on a neighbourhood \(V\) of \(x\), the saturation \(GV\)
satisfies the same bound and \(q(V)=q(GV)\) is a neighbourhood of \(q(x)\).
Apply a finite subcover.
\end{proof}

\subsection{Exhaustion}

Fix \(\gamma_*\in\calE_K(M)\) and let \(\Lambda_*=\Rop(\gamma_*)\).

\begin{definition}[Persistent homotopy]\label{def:persistent-homotopy}
For \(\Lambda\ge\Lambda_*\) put
\[
  \Pi^K_i(\Lambda)=\pi_i\bigl(\calE_{K,\Lambda}(M),\gamma_*\bigr),
\]
with structural maps induced by the inclusions
\(\calE_{K,\Lambda}\hookrightarrow\calE_{K,\Lambda'}\) for
\(\Lambda\le\Lambda'\).
\end{definition}

No tameness or finiteness is assumed; ``persistent homotopy'' here names
this filtered functor and nothing more.  See
\cref{rem:no-barcodes}.

\begin{theorem}[Homotopical exhaustion]\label{thm:homotopical-exhaustion}
For every \(i\ge0\) the inclusions induce an isomorphism
\[
  \colim_{\Lambda\to\infty}\Pi^K_i(\Lambda)
  \ \xrightarrow{\ \cong\ }\
  \pi_i(\calE_K(M),\gamma_*).
\]
The statement holds for closed and for strict sublevels, and in the
\(C^{1,1}\) model.
\end{theorem}

\begin{proof}
Surjectivity: a class is represented by \(F:(S^i,*)\to(\calE_K,\gamma_*)\),
whose image is compact, hence contained in some \(\calE_{K,\Lambda}\) by
\cref{thm:compact-family-bound}.

Injectivity: if two finite-level classes agree in \(\pi_i(\calE_K)\), a
based homotopy \(H:S^i\times[0,1]\to\calE_K\) between representatives has
compact image, hence lies in a single finite sublevel, so the classes agree
there.  For \(i=0\) replace spheres and homotopies by points and paths.  For
strict sublevels enlarge \(\Lambda\) slightly; for the \(C^{1,1}\) model use
\cref{prop:c11-local-bound}.
\end{proof}

\begin{corollary}[Homological exhaustion]\label{cor:homology-exhaustion}
For every ring \(R\) and \(j\ge0\),
\(\colim_\Lambda H_j(\calE_{K,\Lambda};R)\cong H_j(\calE_K;R)\).
\end{corollary}

\begin{proof}
Each singular simplex has compact image, and a chain is a finite sum of
simplices, so the singular chain complex of \(\calE_K\) is the filtered
colimit of those of the sublevels.  Filtered colimits of modules are exact.
\end{proof}

\begin{corollary}[Telescope]\label{cor:telescope}
For any cofinal \(\Lambda_*\le\Lambda_1<\Lambda_2<\cdots\to\infty\) the
canonical map
\(\operatorname{Tel}(\calE_{K,\Lambda_1}\to\calE_{K,\Lambda_2}\to\cdots)
\to\calE_K(M)\)
is a weak homotopy equivalence.
\end{corollary}

\begin{corollary}[Quotient and Borel versions]
\label{cor:quotient-exhaustion}
\Cref{thm:homotopical-exhaustion,cor:homology-exhaustion,cor:telescope} hold
for \(\calU_K(M)\), for \(\mathscr K_K(M)\), and for the Borel constructions
\(EG\times_G\calE_{K,\Lambda}\) associated with a compact Lie group \(G\)
acting on \(M\) by isometries.
\end{corollary}

\begin{proof}
Ropelength is invariant under reparametrisation and ambient isometry, and
the orbit maps are open, so \cref{lem:quotient-descent} applies.  Over a
local trivialisation of \(EG\to BG\) the induced function on the Borel
construction is pulled back from the knot-space factor, hence locally
bounded there too.
\end{proof}

\begin{figure}[t]
\centering
\begin{tikzpicture}[scale=1.0]
\foreach \i/\cx/\w/\h in {1/0/0.62/0.42, 2/2.25/0.88/0.60, 3/4.9/1.14/0.78,
                          4/7.95/1.40/0.96}{
  \draw[rounded corners=7pt,fill=ropesand,draw=ropegrey,line width=0.7pt]
    ({\cx-\w},{-\h}) rectangle ({\cx+\w},{\h});
  \node[font=\scriptsize,below] at (\cx,{-\h-0.1}) {$\calE_{K,\Lambda_\i}$};
}
\foreach \a/\b in {0.68/1.31, 3.19/3.70, 6.10/6.49}{
  \draw[->,>=Stealth,line width=0.8pt] (\a,0) -- (\b,0);}
\draw[->,>=Stealth,line width=0.8pt] (9.42,0) -- (10.25,0);
\node[font=\scriptsize,above] at (9.83,0.08) {$\simeq_w$};
\node[font=\scriptsize,right] at (10.32,0) {$\calE_K$};
\end{tikzpicture}
\caption{\Cref{cor:telescope}: the mapping telescope of any cofinal sequence
of sublevels is weakly equivalent to the whole component.  Filtering by
ropelength therefore refines the classical theory without losing any of it.}
\label{fig:telescope}
\end{figure}

\subsection{What compact capture does and does not give}

\begin{remark}[Honest reading of \cref{thm:homotopical-exhaustion}]
\label{rem:honest-exhaustion}
The exhaustion theorem is soft.  It follows from compactness of spheres and
of \([0,1]\), and it would hold for any function on the knot space that is
locally bounded.  What it does is to license the whole enterprise: it says
that a class of the knot space is \emph{always} visible at some finite
budget, and that two classes visible at finite budget are identified in the
limit only if they are already identified at some finite budget.  Without
it, ``persistent homotopy of a knot space'' would not be known to converge
to anything.  With it, the only remaining questions are quantitative: at
which level, and with how much excess.
\end{remark}

\begin{remark}[Smoothing at a critical level]\label{rem:smoothing}
There is one technical point that the reader should be aware of and that is
not resolved here.  The inclusion
\(\calE_{K,\Lambda}\hookrightarrow\wh{\calE}_{K,\Lambda}\) of the smooth into
the \(C^{1,1}\) sublevel is not known to be a weak equivalence at a fixed
closed level, because an arbitrarily small smoothing of a \(C^{1,1}\) family
may push ropelength across the value \(\Lambda\).  All statements in this
paper are therefore made either for a fixed model, or after
\emph{right-relaxation}, that is after passing to the direct limit over
levels \(\Lambda+\eps\), \(\eps>0\).  For Euclidean strong parameterized families, relative smoothing after positive budget relaxation is proved in \cite{OzawaFinite}. That statement does not assert a fixed-critical-level equivalence or a quotient path-lifting theorem, and its Euclidean formulation is not a claim about every ambient metric.
\end{remark}

\section{Widths, death levels, and the persistence window}
\label{sec:widths}

Exhaustion turns the filtration into numerical data attached to homotopy
classes.  This section defines that data, proves its formal properties, and
then uses symmetry to show that a large part of the \emph{stable} data is
geometrically inexpensive in the round-sphere model.  This is a useful
restriction on where genuinely filtered phenomena can occur, but it does not
eliminate transient finite-level classes.

\subsection{Width}

\begin{definition}[Ropelength width]\label{def:width}
Let \(i\ge1\) and \(\alpha\in\pi_i(\calE_K(M),\gamma_*)\).  Its
\emph{ropelength width} is
\[
  \wid_i^K(\alpha;\gamma_*)
  =\inf\Bigl\{\ \sup_{z\in S^i}\Rop(F(z))\ :\
    F:(S^i,*)\to(\calE_K,\gamma_*),\ [F]=\alpha\Bigr\}.
\]
Equivalently,
\[
  \wid_i^K(\alpha;\gamma_*)
  =\inf\bigl\{\Lambda\ge\Lambda_*:\
    \alpha\in\im[\pi_i(\calE_{K,\Lambda},\gamma_*)\to
                 \pi_i(\calE_K,\gamma_*)]\bigr\}.
\]
Its \emph{excess width} is
\(\exc_i^K(\alpha;\gamma_*)=\wid_i^K(\alpha;\gamma_*)-\Lambda_*\ge0\).
\end{definition}

The two formulations agree because the peak of \(\Rop\) over a representative
family is precisely the smallest level containing it.  No claim is made that
the infimum is attained.

\begin{figure}[t]
\centering
\begin{tikzpicture}[scale=0.88]
% landscape: a based family starts and ends at the basepoint
\def\fam{(0,0.1) .. controls (1.1,0.1) and (1.0,1.75) .. (2.1,1.75)
        .. controls (3.2,1.75) and (3.0,0.45) .. (4.0,0.45)
        .. controls (4.9,0.45) and (4.9,1.2) .. (5.4,1.2)
        .. controls (5.9,1.2) and (5.9,0.1) .. (6.2,0.1)}
\fill[ropesand] \fam -- (6.2,-0.2) -- (0,-0.2) -- cycle;
\draw[line width=0.9pt,ropegrey] \fam;
\draw[ropeblue,dashed,line width=0.8pt] (-0.3,1.75) -- (6.5,1.75);
\node[ropeblue,font=\scriptsize,left] at (-0.3,1.75) {peak};
\draw[ropered,dashed,line width=0.8pt] (-0.3,0.1) -- (6.5,0.1);
\node[ropered,font=\scriptsize,left] at (-0.3,0.1) {$\Lambda_*$};
\filldraw[ropered] (0,0.1) circle (1.6pt);
\filldraw[ropered] (6.2,0.1) circle (1.6pt);
\node[ropered,font=\scriptsize,above right] at (0.05,0.14) {$\gamma_*$};
\node[ropered,font=\scriptsize,above left] at (6.15,0.14) {$\gamma_*$};
\node[font=\scriptsize] at (3.1,-0.55) {a based family $F$ representing $\alpha$};
% right: min over families
\begin{scope}[shift={(8.4,0)}]
\draw[->,>=Stealth,line width=0.9pt] (0,-0.2) -- (0,2.1);
\node[rotate=90,font=\scriptsize] at (-0.62,1.0) {peak $\Rop$};
\foreach \x/\y in {0.5/1.75,1.0/1.45,1.5/1.25,2.0/1.12,2.5/1.06}{
  \filldraw[ropeblue] (\x,\y) circle (1.5pt);}
\draw[ropered,dashed,line width=0.8pt] (0,1.0) -- (3.0,1.0);
\node[ropered,font=\scriptsize,right] at (3.02,1.0) {$\wid_i(\alpha)$};
\node[font=\scriptsize,below] at (1.5,-0.3) {families representing $\alpha$};
\end{scope}
\end{tikzpicture}
\caption{Width is a min--max quantity: each based family representing a class
starts and ends at the basepoint \(\gamma_*\) and has a peak ropelength, and the width is the infimum of the peaks.  The excess
width \(\exc_i=\wid_i-\Lambda_*\) measures how much rope the class costs
\emph{beyond} the basepoint configuration.}
\label{fig:width}
\end{figure}

\begin{proposition}[Width is a non-Archimedean seminorm]
\label{prop:width-properties}
For every \(i\ge1\):
\begin{enumerate}[label=\textup{(\roman*)},leftmargin=2.6em]
\item \(\Lambda_*\le\wid_i^K(\alpha;\gamma_*)<\infty\);
\item \(\wid_i^K(1;\gamma_*)=\Lambda_*\);
\item \(\wid_i^K(\alpha^{-1};\gamma_*)=\wid_i^K(\alpha;\gamma_*)\);
\item \(\wid_i^K(\alpha\beta;\gamma_*)\le
  \max\{\wid_i^K(\alpha;\gamma_*),\wid_i^K(\beta;\gamma_*)\}\).
\end{enumerate}
Hence \(\exc_i^K(-;\gamma_*)\) is a non-Archimedean group seminorm.  For
\(i\ge2\) read the group additively.
\end{proposition}

\begin{proof}
(i) The basepoint lies in the image of every based representative, whence
the lower bound; finiteness is \cref{thm:homotopical-exhaustion}.  (ii) Use
the constant map.  (iii) Precompose with an orientation-reversing self-map
of \(S^i\); the image, hence the peak, is unchanged.  (iv) Given
representatives with peaks within \(\eps\) of the two widths, concatenate
(for \(i=1\)) or use the pinch map \(S^i\to S^i\vee S^i\) (for \(i\ge2\));
the image of the resulting representative is the union of the two images, so
its peak is at most the larger of the two peaks.  Let \(\eps\to0\).
\end{proof}

\begin{definition}[Stable filtration and entry level]
Put
\[
  F_\Lambda\pi_i(\calE_K,\gamma_*)
  =\im\bigl[\pi_i(\calE_{K,\Lambda},\gamma_*)\to
    \pi_i(\calE_K,\gamma_*)\bigr].
\]
These subgroups increase with \(\Lambda\) and exhaust
\(\pi_i(\calE_K,\gamma_*)\); the width of \(\alpha\) is its entry level into
this filtration.  Put
\[
  \sigma_i(K;\gamma_*)=\sup_{\alpha\in\pi_i(\calE_K,\gamma_*)}
    \wid_i^K(\alpha;\gamma_*)\in[\Lambda_*,\infty],
\]
the \emph{stabilisation level} in degree \(i\): above it, every stable class
is already realised.
\end{definition}

\subsection{Symmetry is free}

The following observation is elementary but it decides the shape of the
theory, so we state it separately.

\begin{theorem}[Isometry orbits have zero excess width]
\label{thm:orbit-width}
Let \(G\) act on \(M\) by isometries preserving \(\calE_K(M)\), let
\(\gamma_*\in\calE_K(M)\) and let
\[
  o:G\longrightarrow\calE_K(M),\qquad g\longmapsto g\circ\gamma_* ,
\]
be the orbit map, with image \(O=G\cdot\gamma_*\).  Then
\(\Rop\equiv\Lambda_*\) on \(O\), and consequently:
\begin{enumerate}[label=\textup{(\roman*)},leftmargin=2.6em]
\item every class in \(\im[o_*:\pi_i(G,e)\to\pi_i(\calE_K,\gamma_*)]\) has
excess width \(0\);
\item more generally, if the inclusion \(O\hookrightarrow\calE_K(M)\) is
surjective on \(\pi_i\) for some \(i\ge1\), then
\(\wid_i^K(\alpha;\gamma_*)=\Lambda_*\) for \emph{every}
\(\alpha\in\pi_i(\calE_K,\gamma_*)\), that is \(\exc_i\equiv0\) and
\(\sigma_i(K;\gamma_*)=\Lambda_*\).
\end{enumerate}
\end{theorem}

\begin{proof}
Isometries preserve length and thickness, so \(\Rop(g\circ\gamma_*)=
\Rop(\gamma_*)=\Lambda_*\) for all \(g\in G\); thus \(O\) lies in the level
set \(\{\Rop=\Lambda_*\}\).  (Only this property of \(O\) is used, so the
statement holds verbatim in any of the models of \cref{sec:knotspaces} and
for any filtration function that is constant on \(G\)-orbits.)  A based representative of a class in the image
of \(\pi_i(O)\) has all its values in \(O\), hence peak ropelength
\(\Lambda_*\), so its width is at most \(\Lambda_*\); by
\cref{prop:width-properties}(i) it equals \(\Lambda_*\).  Statement (ii) is
the same argument applied to every class.
\end{proof}

\begin{corollary}[The Gramain class costs no rope]
\label{cor:gramain-width}
Let \(f\) be a long knot agreeing outside a compact set with a fixed line
through the origin, and let \(\Rop_c\) denote ropelength computed after
one-point compactification into the round \(S^3\)
\textup{(\cref{subsec:longknots})}.  Then
\[
  \wid_1([G_f];f)=\Rop_c(f),
  \qquad\text{that is}\qquad
  \exc_1([G_f];f)=0 .
\]
In particular, for a nontrivial long knot the excess width seminorm
\(\exc_1\) on \(\pi_1(\calK_f,f)\) is \emph{not} a norm: it vanishes on the
nontrivial class \([G_f]\) of \cref{thm:gramain}.
\end{corollary}

\begin{proof}
Rotation about the axis is an element of \(SO(3)\subset SO(4)\) fixing the
two poles of the compactification, hence acts on \(S^3\) by isometries.  The
Gramain loop is the orbit of \(f\) under the corresponding circle subgroup,
so \cref{thm:orbit-width}(i) applies with \(G=S^1\).
\end{proof}

\begin{remark}
\Cref{cor:gramain-width} settles, negatively, a question that the formal
theory leaves open: excess width could a priori have been a norm,
detecting every nontrivial class.  It is not.  The right intuition is that
\emph{width measures the geometric cost of a family beyond what rigid
motions already provide}, and rigid motions provide a great deal.
\end{remark}

\subsection{A vanishing theorem}

For closed knots in \(S^3\) the situation is more dramatic, because
Hatcher's theorems say that in the best understood cases the knot space is
\emph{nothing but} an orbit of the isometry group.

\begin{definition}[Orbit realisation]\label{hyp:orbit-retract}
A knot type \(K\subset S^3\) is \emph{orbit realised} if there is a smooth
representative \(J_0\) such that the inclusion of the orbit
\(SO(4)\cdot J_0\) into \(\mathscr K_K(S^3)\) is surjective on \(\pi_i\) for
every \(i\ge0\).
\end{definition}

\begin{theorem}[Vanishing of excess width]\label{thm:vanishing-width}
Let \(K\) be orbit realised by \(J_0\), and take \(J_0\) as basepoint.  Then
for every \(i\ge1\) and every
\(\alpha\in\pi_i(\mathscr K_K(S^3),J_0)\),
\[
  \wid_i(\alpha;J_0)=\Rop_{S^3}(J_0),
  \qquad
  \exc_i(\alpha;J_0)=0,
  \qquad
  \sigma_i(K;J_0)=\Rop_{S^3}(J_0).
\]
Consequently the persistent homotopy of the spherical ropelength filtration
of \(\mathscr K_K(S^3)\) surjects onto its stable value at the single finite
level \(\Rop_{S^3}(J_0)\), in every degree simultaneously.
\end{theorem}

\begin{proof}
Immediate from \cref{thm:orbit-width}(ii) with \(G=SO(4)\), which acts on
the round \(S^3\) by isometries.
\end{proof}

\begin{corollary}\label{cor:vanishing-torus-hyperbolic}
Every nontrivial torus knot and every hyperbolic knot in the round
\(S^3\) is orbit realised by a maximally symmetric placement as in
\cref{thm:hatcher-models}\textup{(i)},\textup{(ii)}.  With such a placement
as basepoint, every \emph{stable} positive-degree homotopy class has zero
excess width.
\end{corollary}

\begin{proof}
In cases (i) and (ii) of \cref{thm:hatcher-models} the homotopy equivalence
is induced by the orbit map of \(SO(4)\) through the symmetric placement
\(J_0\), whose stabiliser is \(G_K\) respectively \(\Gamma_K\); the
inclusion of the orbit is therefore a homotopy equivalence and in particular
\(\pi_i\)-surjective.  Apply \cref{thm:vanishing-width}.  (See
\cref{rem:linearisation} for the input needed in the hyperbolic case.)
\end{proof}

\begin{figure}[t]
\centering
\begin{tikzpicture}[scale=1.0]
\draw[fill=ropesand,draw=ropegrey,line width=0.7pt,rounded corners=18pt]
  (0,-0.35) rectangle (8.6,2.5);
\node[font=\scriptsize,below right] at (0.08,-0.4) {$\mathscr K_K(S^3)$};
% level set = orbit
\draw[ropeblue,line width=1.2pt,fill=ropelight]
  (3.1,1.05) ellipse (2.35 and 0.72);
\draw[ropeblue,line width=1.2pt,fill=ropesand]
  (3.1,1.05) ellipse (1.1 and 0.3);
\node[ropeblue,font=\scriptsize,align=center] at (3.1,2.1)
  {orbit $SO(4)\cdot J_0$: one level set of $\Rop_{S^3}$};
\filldraw[ropered] (0.95,1.28) circle (1.6pt);
\node[ropered,font=\scriptsize,left] at (0.9,1.3) {$J_0$};
% an arbitrary family and its deformation into the orbit
\draw[ropered,line width=1pt,dashed] (7.2,1.05) ellipse (0.75 and 0.8);
\node[ropered,font=\scriptsize,align=center] at (7.2,1.05) {a family\\ $F$};
\draw[->,>=Stealth,ropered,line width=0.9pt] (6.4,1.05) -- (5.55,1.05);
\node[ropered,font=\scriptsize,above] at (5.97,1.1) {deform};
\end{tikzpicture}
\caption{\Cref{thm:vanishing-width}.  If the terminal knot space is
homotopically represented by one isometry orbit, any family can be deformed
into the orbit, and every \emph{stable}
homotopy class has a representative in a single ropelength level.  This is a
statement about the image in the terminal homotopy group; finite-level
transient classes can still be present.}
\label{fig:orbit-level}
\end{figure}

\begin{remark}[How to read a vanishing theorem]
\label{rem:reading-vanishing}
\Cref{thm:vanishing-width} is a negative result, and it is the most useful
statement in this section.  It says that for torus and hyperbolic knots the
stable excess widths vanish in the stated round-sphere model and at the specified orbit basepoint.  It
does \emph{not} say that the filtration is uninteresting for these knots:
the levels \emph{below} \(\Rop_{S^3}(J_0)\) may well carry finite-level
classes that later die (\cref{subsec:death}), and degree zero is untouched
by the theorem, since \(\mathscr K_K(S^3)\) is connected while
\(Y_\Lambda(K)\) need not be.
\end{remark}

\subsection{Where positive stable width might live}

By \cref{thm:hatcher-models}(iii), satellite models contain
\(X_K\)-directions that are not generated solely by ambient isometries.
There is therefore no formal reason for families in these directions to
have constant ropelength.  This makes satellites natural test cases for
positive \emph{stable} excess width.  It is a motivation, not a theorem that
positive width must occur there.

\begin{problem}[Satellite widths]\label{prob:satellite-width}
Let \(K\) be a satellite knot and let \(\alpha\) be a homotopy class
coming from a non-isometric direction of Hatcher's parameter space \(X_K\)
(for example a torus factor in a nested satellite).  Compute or estimate
its excess width in the round-sphere filtration.  Is it positive for some
satellite knot?
\end{problem}

\begin{problem}[The first positive-degree phenomenon]
\label{prob:first-positive-width}
Exhibit a knot type \(K\), a degree \(i\ge1\), and a class
\(\alpha\in\pi_i(\calE_K(M))\) with \(\exc_i(\alpha)>0\); or prove that
excess widths vanish identically for all knot types.
\end{problem}

We regard \cref{prob:first-positive-width} as the central concrete question
raised by this paper.  As of this writing, \emph{no} example of a positive
excess width is known in any degree \(i\ge1\), while in degree zero the
corresponding phenomenon is a theorem for links
(\cref{subsec:gordian}).

\subsection{Death levels and transient classes}
\label{subsec:death}

\begin{definition}[Death level]\label{def:death}
For \(a\in\pi_i(\calE_{K,\Lambda},\gamma_*)\) put
\[
  \mathfrak d(a)=\inf\{\Lambda'\ge\Lambda:\ a\mapsto1\ \text{in}\
  \pi_i(\calE_{K,\Lambda'},\gamma_*)\},
\]
with \(\mathfrak d(a)=\infty\) if no such level exists.
\end{definition}

\begin{corollary}[Finite death criterion]\label{cor:finite-death}
A finite-level class \(a\) maps to the identity in \(\pi_i(\calE_K)\) if and
only if \(\mathfrak d(a)<\infty\).  In particular, if \(\calE_K\) is
aspherical then every finite-level class of degree \(i\ge2\) has finite
death level.
\end{corollary}

\begin{proof}
Injectivity in \cref{thm:homotopical-exhaustion} gives the forward
implication; the converse is trivial.
\end{proof}

\begin{corollary}[Transient higher homotopy of long knots]
\label{cor:transient-higher}
Let \(f\) be any long knot and \(i\ge2\).  Every class in
\(\pi_i(\calK_{f,\Lambda},f)\) dies at a finite level.
\end{corollary}

\begin{proof}
\(\calK_f\) is a \(K(\pi,1)\) by \cref{thm:budney-long}, so the stable group
vanishes; apply \cref{cor:finite-death} and
\cref{thm:longknot-exhaustion}.
\end{proof}

This is a prediction of a purely geometric kind, and it is the cheapest source of
possible examples for \cref{prob:first-positive-width}: a nonzero class in
\(\pi_i(\calK_{f,\Lambda})\), \(i\ge2\), would be a family of thick knots
that cannot be filled below level \(\Lambda\) but can be filled above it, so
its birth and death levels bracket a finite-scale phenomenon
(\cref{fig:barcode}).

\begin{figure}[t]
\centering
\begin{tikzpicture}[x=0.8cm,y=0.55cm]
\draw[->,>=Stealth,line width=0.8pt] (0,-0.6) -- (9.2,-0.6);
\node[font=\scriptsize,right] at (9.25,-0.6) {$\Lambda$};
\foreach \x in {1,2.4}{\draw[line width=0.5pt] (\x,-0.6) -- (\x,-0.85);}
\node[font=\scriptsize,below] at (1,-0.9) {$m_K$};
\node[font=\scriptsize,below] at (2.4,-0.9) {$\sigma_i(K)$};
\draw[ropegrey,dotted,line width=0.6pt] (2.4,-0.6) -- (2.4,1.9);
% bars
\draw[ropeblue,line width=3.4pt] (1.2,3.0) -- (8.9,3.0);
\node[font=\scriptsize,left] at (1.1,3.0) {stable class};
\draw[ropegreen,line width=3.4pt] (2.4,1.9) -- (8.9,1.9);
\node[font=\scriptsize,left] at (2.3,1.9) {stable class};
\draw[ropered,line width=3.4pt] (1.6,0.8) -- (4.6,0.8);
\node[font=\scriptsize,left] at (1.5,0.8) {transient};
\draw[ropered,line width=3.4pt] (3.4,-0.1) -- (5.6,-0.1);
\node[font=\scriptsize,left] at (3.3,-0.1) {transient};
\draw[ropegrey,dashed,line width=0.7pt] (8.9,-0.4) -- (8.9,3.5);
\node[font=\scriptsize,above,ropegrey] at (8.9,3.5) {$\Lambda\to\infty$};
\end{tikzpicture}
\caption{The picture the invariants are trying to draw: classes that survive
to the limit have a birth level (their width) and no death; transient
classes have both.  The stabilisation level \(\sigma_i(K)\) is the largest
birth level of a stable class.  \Cref{cor:transient-higher} says that for long knots
\emph{all} classes of degree \(\ge2\) are transient.  This picture is a
heuristic, not a theorem: see \cref{rem:no-barcodes}.}
\label{fig:barcode}
\end{figure}

\begin{remark}[Scope of tameness]\label{rem:no-barcodes}
Homotopy persistence is nonabelian in degree one and is distinct from
homology persistence over a field. The marked weak model in
\cref{thm:marked-finite-factorization} has finite-rank homology maps across
every positive gap. This does not prove pointwise finite dimensionality,
a CW description at a closed level, or the corresponding assertion for
strong or coarse quotient models. Widths and death levels require none of
these additional hypotheses.
\end{remark}

\subsection{A dictionary}

\begin{table}[htbp]
\small
\begin{tabular}{@{}p{5.05cm}p{7.05cm}@{}}
\toprule
\textbf{geometric side} & \textbf{topological side}\\
\midrule
rope budget \(\Lambda\) & sublevel \(\calE_{K,\Lambda}\) or
  \(Y_\Lambda(K)\)\\
tight (ideal) configurations & the minimum level set \(I(K)=Y_{m_K}(K)\)\\
two configurations cannot be connected within budget &
  they lie in different components of \(Y_\Lambda(K)\)\\
level at which connected components merge & connected merge height \(\mu_K\)\\
min--max budget for an actual connecting path & path merge \(\mu_K^{\rm path}\)\\
extra rope needed to perform a cyclic motion & excess width \(\exc_1\)\\
a family of motions that only exists with slack &
  a class born at a finite level\\
slack that becomes unnecessary & a class with finite death level\\
budget above which nothing new appears & stabilisation level \(\sigma_i(K)\)\\
\bottomrule
\end{tabular}
\caption{The dictionary used throughout the paper.}
\label{tab:dictionary}
\end{table}

\section{Degree zero: thick isotopy, barriers, and merge ultrametrics}
\label{sec:degreezero}

This section contains the part of the theory where something is actually
known to happen.  We work in the normalised Euclidean model, where length is
a function and the questions are about connectivity of its sublevels.

\subsection{The normalised model}

\begin{definition}[Normalised sublevels]\label{def:Y}
Let \(K\) be a tame knot type.  Represent configurations by constant-speed
\(C^{1,1}\) embeddings \(\gamma:S^1=\R/2\pi\Z\to\R^3\) and identify two of
them when they differ by an orientation-preserving Euclidean isometry and a
rotation or reflection of the source circle.  For \(m_K\le\Lambda<\infty\)
put
\[
  Y_\Lambda(K)=\bigl\{[\gamma]:\ [\gamma]=K,\
  \Thi(\gamma)\ge1,\ \Len(\gamma)\le\Lambda\bigr\},
\]
\[
  Y_\infty(K)=\bigcup_{\Lambda}Y_\Lambda(K),
\]
topologised as a quotient of the constant-speed \(C^1\) topology.  The
\emph{ideal stratum} is \(I(K)=Y_{m_K}(K)\), nonempty by
\cref{thm:cks-existence}.
\end{definition}

\subsection{Sublevels versus fixed-length physical ropes}

For \(L\ge m_K\), let
\[
  P_L(K)=\{[\gamma]\in Y_L(K):\Len(\gamma)=L\}
\]
be the exact-length shell.  This is the natural configuration space for a
rope of fixed core length \(L\) carrying an embedded unit-radius tube: a
path in \(P_L(K)\) is precisely a fixed-length thick isotopy after
quotienting by rigid motions.

\begin{proposition}[Exact-length retraction]\label{prop:fixed-length-retract}
For every \(L\ge m_K\), the inclusion \(P_L(K)\hookrightarrow Y_L(K)\) is a
strong deformation retract.  In particular two points of \(P_L(K)\) are
joined by a path in \(Y_L(K)\) if and only if they are joined by a
fixed-length thick isotopy inside \(P_L(K)\).
\end{proposition}

\begin{proof}
For a representative \(\gamma\) of length \(\ell\le L\), scale about its
arclength centroid by
\[
 s_t(\gamma)=(1-t)+t\frac{L}{\ell},\qquad 0\le t\le1.
\]
The resulting curve has length \(s_t\ell\le L\) and thickness
\(s_t\Thi(\gamma)\ge1\), since \(s_t\ge1\).  Scaling about the arclength
centroid is equivariant under Euclidean isometries and is unaffected by the
allowed source reparametrisations, so this construction descends
continuously to the quotient.  It fixes every point with \(\ell=L\), and at
\(t=1\) its image lies in \(P_L(K)\).  Hence it is a strong deformation
retraction.
\end{proof}

\begin{remark}[Why this simple observation matters]\label{rem:physical-shell}
A path in the \emph{sublevel} is sometimes described informally as a
physical isotopy, but a physical rope has fixed length.  The proposition
justifies that language only at the level of \emph{path components}: any
sublevel path between length-\(L\) endpoints can be pushed, relative to its
endpoints, onto the exact-length shell.  It says nothing about connected
sets that are not paths.  This is exactly the distinction used below.
\end{remark}

\begin{convention}[Connected versus path components]\label{conv:conn-pi0}
Because the distinction is essential here, we use
\(\Conn(X)\) for the set of \emph{connected components} of a space \(X\),
and reserve the standard notation \(\pi_0(X)\) for its set of
\emph{path components}.  No equality between these two sets is assumed for
finite ropelength sublevels.
\end{convention}

\subsection{Compactness}

\begin{theorem}[Compactness of normalised sublevels]
\label{thm:knot-sublevel-compact}
For every knot type \(K\) and every finite \(\Lambda\ge m_K\), the space
\(Y_\Lambda(K)\) is compact and Hausdorff.
\end{theorem}

\begin{proof}
Translate each representative so that its arclength centroid is the origin.
Since a closed curve of length \(\le\Lambda\) has diameter \(\le\Lambda/2\),
all centred representatives lie in a fixed ball.

Let \(v=\Len(\gamma)/2\pi\) be the constant speed.  By
\cref{prop:unknot-2pi}, \(2\pi\le\Len(\gamma)\le\Lambda\), so
\(1\le v\le\Lambda/2\pi\).  Thickness at least \(1\) gives \(|\kappa|\le1\)
almost everywhere, hence in the angular parameter
\[
  |\gamma'|\equiv v\le\frac{\Lambda}{2\pi},
  \qquad
  \operatorname{Lip}(\gamma')\le v^2\|\kappa\|_\infty
  \le\Bigl(\frac{\Lambda}{2\pi}\Bigr)^{2}.
\]
So the centred constant-speed representatives, together with their first
derivatives, are uniformly bounded and equicontinuous.  By
Arzel\`a--Ascoli every sequence has a subsequence converging in \(C^1\) to
some \(\gamma_\infty\).

By \cref{rem:three-point} the condition \(\Thi\ge1\) is closed under uniform
convergence, so \(\Thi(\gamma_\infty)\ge1\); in particular
\(\gamma_\infty\) is embedded.  Length is continuous in \(C^1\), so
\(\Len(\gamma_\infty)\le\Lambda\).  Finally, curves of thickness at least
\(1\) that are sufficiently \(C^1\)-close are ambient isotopic
\cite[\S2]{CKS}, so \(\gamma_\infty\) has knot type \(K\).

Thus the set of centred constant-speed representatives is compact in
\(C^1\).  After centring and constant-speed normalisation the residual
symmetries form the compact group \(SO(3)\times O(2)\), and \(Y_\Lambda(K)\)
is exactly its orbit space; a quotient of a compact Hausdorff space by a
compact group is compact Hausdorff.
\end{proof}

\begin{remark}[One component is essential]\label{rem:one-component}
The hypothesis that \(K\) is a \emph{knot} and not a link is doing real
work.  For a split link, one sublink may drift to infinity while length and
thickness stay bounded, so the analogous quotient is not compact and merge
levels need not be attained.  This is exactly the situation in which the
Gordian split link of Coward and Hass \cite{CowardHass} lives, and it is why
their theorem and \cref{thm:merge-ultrametric} are complementary rather than
overlapping statements.
\end{remark}

\subsection{Components, quasicomponents, and right continuity}

\begin{lemma}\label{lem:components-quasi}
In a compact Hausdorff space the connected components coincide with the
quasicomponents; that is, two points lie in the same component if and only
if they cannot be separated by a clopen set.  Moreover a nested intersection
of compact connected sets is compact and connected.
\end{lemma}

\begin{proof}
Both are standard; see for instance \cite[Theorem~6.1.23]{Engelking} and
\cite[Theorem~6.1.18]{Engelking}.
\end{proof}

\begin{corollary}[Right continuity of connected components]\label{cor:right-continuity}
Let \(x,y\in Y_\Lambda(K)\).  If \(x\) and \(y\) lie in the same connected
component of \(Y_{\Lambda+\eps}(K)\) for every \(\eps>0\), then they lie in
the same connected component of \(Y_\Lambda(K)\).  Equivalently, the map
\[
  \Conn\bigl(Y_\Lambda(K)\bigr)\longrightarrow
  \varprojlim_{\eps>0}\ \im\bigl[\Conn(Y_\Lambda(K))\to
  \Conn(Y_{\Lambda+\eps}(K))\bigr]
\]
is injective.
\end{corollary}

\begin{proof}
Let \(C_n\) be the component of \(Y_{\Lambda+1/n}(K)\) containing \(x\); by
hypothesis \(y\in C_n\).  Each \(C_n\) is closed in a compact space, hence
compact, and \(C_{n+1}\subset C_n\).  By \cref{lem:components-quasi} the
intersection \(C_\infty=\bigcap_n C_n\) is compact, connected and contains
\(\{x,y\}\).  Since \(\bigcap_n Y_{\Lambda+1/n}(K)=Y_\Lambda(K)\), we have
\(C_\infty\subset Y_\Lambda(K)\), so \(x\) and \(y\) lie in one component of
\(Y_\Lambda(K)\).
\end{proof}

\subsection{The merge function}

\begin{definition}[Merge level]\label{def:merge}
For \(x,y\in Y_\infty(K)\) put
\[
  \mu_K(x,y)=\inf\bigl\{\Lambda:\ x\ \text{and}\ y\ \text{are connected in}\ Y_\Lambda(K)\bigr\}.
\]
More generally, for connected subsets \(A,B\subset Y_\infty(K)\) we write
\(\mu_K(A,B)\) for the infimum of the levels \(\Lambda\) at which
\(A\cup B\) is contained in one connected component of \(Y_\Lambda(K)\);
this agrees with the previous formula when \(A\) and \(B\) are points.
\end{definition}

\begin{lemma}[Smoothing one positive-thickness knot]
\label{lem:single-knot-smoothing}
Let \(\gamma:S^1\to\R^3\) be a \(C^{1,1}\) embedding of positive
thickness.  Then \(\gamma\) can be joined to a smooth embedding of the same
knot type by a continuous path of \(C^{1,1}\) embeddings along which length
is uniformly bounded and thickness is uniformly bounded away from zero.
\end{lemma}

\begin{proof}
Choose constant-speed coordinates and periodically mollify \(\gamma\).
The smooth curves \(\gamma_\eps\) converge to \(\gamma\) in \(C^1\), and
the Lipschitz constants of their tangent fields remain uniformly bounded.
For small \(\eps\), the straight interpolation between \(\gamma\) and
\(\gamma_\eps\) stays \(C^1\)-close to \(\gamma\) with the same type of
uniform tangent-Lipschitz bound.  This gives a uniform positive lower bound
for the local curvature-radius contribution to thickness.  For the
nonlocal contribution, remove a small neighbourhood of the diagonal in
\(S^1\times S^1\); on the remaining compact set, positive reach gives a
positive separation that persists under sufficiently small \(C^0\)
perturbation.  The excluded near-diagonal pairs are controlled by the local
graph estimate and the tangent-Lipschitz bound.  Thus the whole interpolation
has thickness bounded below by a positive constant, and length is bounded by
\(C^1\)-closeness.  The path therefore remains embedded and in the same knot
type.  This is the standard smoothing mechanism behind the compactness
theory of thick curves; compare \cite{CKS,DDS}.
\end{proof}

\begin{lemma}[Finiteness]\label{lem:finiteness-merge}
\(\mu_K(x,y)<\infty\) for all \(x,y\in Y_\infty(K)\).
\end{lemma}

\begin{proof}
By \cref{lem:single-knot-smoothing}, each endpoint can first be connected,
after a finite dilation if necessary, to a smooth representative through a
path with thickness at least one and finite length bound.  The two smooth
representatives have the same knot type, hence are joined by a smooth
isotopy.  The parameter interval is compact, so
\cref{thm:compact-family-bound} gives a positive lower bound \(\delta\) for
thickness and a finite upper bound \(L\) for length along this isotopy.
Dilating the isotopy by \(\max\{1,\delta^{-1}\}\) puts it inside one finite
normalised sublevel.  Dilation paths at the ends and the two smoothing paths
then concatenate to a path from \(x\) to \(y\) in \(Y_{\Lambda'}(K)\) for
some finite \(\Lambda'\).  Hence both the connected and path merge levels
are finite.
\end{proof}

\begin{theorem}[The merge function]\label{thm:merge-ultrametric}
For all \(x,y,z\in Y_\infty(K)\):
\begin{enumerate}[label=\textup{(\roman*)},leftmargin=2.6em]
\item \(\mu_K(x,x)=\Len(x)\) and \(\mu_K(x,y)=\mu_K(y,x)<\infty\);
\item \(\mu_K(x,z)\le\max\{\mu_K(x,y),\mu_K(y,z)\}\);
\item \emph{(attainment)} \(x\) and \(y\) lie in one connected
component of \(Y_{\Lambda}(K)\) already for
\(\Lambda=\mu_K(x,y)\).
\end{enumerate}
Consequently, on the ideal stratum, where \(\Len\equiv m_K\), the function
\[
  d_K(A,B)=\mu_K(A,B)-m_K
\]
is a genuine finite-valued ultrametric on the set \(\Conn(I(K))\) of
connected components of \(I(K)\); in particular \(d_K(A,B)=0\) if and only
if \(A=B\).
\end{theorem}

\begin{proof}
(i) and (ii) are formal from monotonicity of the filtration, together with
\cref{lem:finiteness-merge}.  (iii) is \cref{cor:right-continuity} applied
at \(\Lambda=\mu_K(x,y)\), noting that \(x\) and \(y\) lie in one component
of \(Y_{\mu_K(x,y)+\eps}\) for every \(\eps>0\) by definition of the
infimum.  For the last statement, if \(A\ne B\) are distinct components of
\(I(K)=Y_{m_K}(K)\) then by (iii) \(\mu_K(A,B)>m_K\), so \(d_K(A,B)>0\);
symmetry and the ultrametric inequality are (i) and (ii).
\end{proof}

\begin{corollary}[Positive attained barrier]\label{cor:positive-barrier}
Let \(A\ne B\) be connected components of \(I(K)\).  Then every path in
\(Y_\infty(K)\) from a point of \(A\) to a point of \(B\) attains length at
least \(m_K+d_K(A,B)>m_K\) somewhere; and at the level
\(m_K+d_K(A,B)\) itself, \(A\) and \(B\) already lie in one connected
component of the sublevel.
\end{corollary}

\begin{proof}
A path from \(A\) to \(B\) whose length stayed below \(\lambda\) would
place \(A\) and \(B\) in one connected component of \(Y_\lambda(K)\), so
\(\lambda\ge\mu_K(A,B)\).  The second assertion is
\cref{thm:merge-ultrametric}(iii).
\end{proof}

\begin{remark}[Connected merge versus path merge]\label{rem:path-vs-conn}
The merge function \(\mu_K\) of \cref{def:merge} uses \emph{connected}
components.  This choice is what makes the attainment argument work: a
nested intersection of compact connected sets is connected, but need not be
path connected.  The physically relevant companion quantity is
\[
  \mu^{\mathrm{path}}_K(x,y)
  =\inf\{\Lambda:\text{there is a path from }x\text{ to }y
                    \text{ in }Y_\Lambda(K)\}.
\]
For a subset \(A\), define \(\mu_K^{\mathrm{path}}(x,A)\) by allowing the
second endpoint to be any point of \(A\).  Then
\[
  \mu_K(x,y)\le \mu_K^{\mathrm{path}}(x,y)<\infty.
\]
Finiteness follows from \cref{lem:finiteness-merge}, and the strong triangle
inequality follows by concatenating paths.  Thus, on the path components of the ideal stratum, path merge gives a
pseudo-ultrametric after subtracting the minimum level, but it need not
separate distinct path components: two path components could in principle
merge only in the limit from the right.  Both issues would disappear if the
relevant sublevels were locally path connected, and
\cref{subsec:lpc} proves that they are, after an arbitrarily small right
relaxation of the budget.  \Cref{thm:merge-equal} then gives
\(\mu_K=\mu_K^{\mathrm{path}}\) unconditionally.  Whether the path infimum is
\emph{attained} is a finer question and remains open; see
\cref{rem:attainment-open}.

By \cref{prop:fixed-length-retract}, this distinction has a concrete
meaning.  If \(x,y\in P_L(K)\), then they are physically isotopic using the
same rope exactly when they are path connected in \(Y_L(K)\).  Mere
connectedness of \(Y_L(K)\) is therefore an analytic lower-bound device,
not by itself a physical deformation.
\end{remark}

\begin{table}[htbp]
\small
\begin{tabular}{@{}p{3.1cm}p{4.05cm}p{4.8cm}@{}}
\toprule
 & \textbf{connected merge \(\mu_K\)} &
   \textbf{path merge \(\mu_K^{\rm path}\)}\\
\midrule
Meaning & topology of compact sublevels & deformation paths / min--max budget\\
Finite? & yes & yes\\
Strong triangle inequality? & yes & yes\\
Infimum attained? & yes & not known in general\\
Separates components at zero excess? & yes, for connected ideal components &
yes (\cref{cor:ultrametric-path})\\
Relation & \multicolumn{2}{@{}p{8.7cm}@{}}{\(\mu_K=\mu_K^{\rm path}\)
(\cref{thm:merge-equal}); only the left infimum is known to be attained.}\\
\bottomrule
\end{tabular}
\caption{Two degree-zero merge notions.  They agree as numbers
(\cref{thm:merge-equal}); keeping them separate is still useful, because the
left column has the better compactness theorem and only its infimum is known
to be attained.}
\label{tab:two-merges}
\end{table}

\begin{figure}[t]
\centering
\begin{tikzpicture}[x=1cm,y=1cm]
% dendrogram
\draw[->,>=Stealth,line width=0.8pt] (-0.6,-0.15) -- (-0.6,3.1);
\node[rotate=90,font=\scriptsize] at (-1.7,1.5) {level $\Lambda$};
\foreach \y/\lab in {0.35/$m_K$,1.35/$\mu_{AB}$,2.55/$\mu_{AC}$}{
  \draw[ropegrey,dashed,line width=0.5pt] (-0.6,\y) -- (5.6,\y);
  \node[font=\scriptsize,left] at (-0.65,\y) {\lab};}
% leaves
\foreach \x/\lab in {0.4/$A$,1.6/$B$,3.2/$C$,4.9/$D$}{
  \draw[ropeblue,line width=1.1pt] (\x,0.35) -- (\x,1.35);
  \filldraw[ropered] (\x,0.35) circle (1.7pt);
  \node[font=\scriptsize,below] at (\x,0.28) {\lab};}
% merges
\draw[ropeblue,line width=1.1pt] (0.4,1.35) -- (1.6,1.35);
\draw[ropeblue,line width=1.1pt] (1.0,1.35) -- (1.0,2.55);
\draw[ropeblue,line width=1.1pt] (3.2,1.35) -- (3.2,1.95) -- (4.9,1.95) -- (4.9,1.35);
\draw[ropeblue,line width=1.1pt] (4.05,1.95) -- (4.05,2.55);
\draw[ropeblue,line width=1.1pt] (1.0,2.55) -- (4.05,2.55);
\draw[ropeblue,line width=1.1pt] (2.5,2.55) -- (2.5,3.0);
\node[font=\scriptsize,above] at (2.5,3.0) {all merged};
\end{tikzpicture}
\hspace{0.4cm}
\begin{tikzpicture}[x=1cm,y=1cm]
\node[font=\scriptsize] at (1.2,2.5) {ultrametric triangle};
\filldraw[ropered] (0,0) circle (1.6pt) node[below left,font=\scriptsize] {$A$};
\filldraw[ropered] (2.4,0) circle (1.6pt) node[below right,font=\scriptsize] {$B$};
\filldraw[ropered] (1.2,1.85) circle (1.6pt) node[above,font=\scriptsize] {$C$};
\draw[ropeblue,line width=1pt] (0,0) -- (2.4,0) node[midway,below,font=\scriptsize] {$d(A,B)$};
\draw[ropeblue,line width=1pt] (0,0) -- (1.2,1.85) node[midway,above left,font=\scriptsize] {$d(A,C)$};
\draw[ropeblue,line width=1pt] (2.4,0) -- (1.2,1.85) node[midway,above right,font=\scriptsize] {$d(B,C)$};
\node[font=\scriptsize,align=center] at (1.2,-0.75)
 {two of the three sides\\ are equal and largest};
\end{tikzpicture}
\caption{\Cref{thm:merge-ultrametric}.  Left: the components of the ideal
stratum, drawn as leaves, merge as the budget rises; compactness makes each
merge \emph{happen at} its level rather than only above it, so the diagram
is a true dendrogram, with \(\mu_{AB}=\mu_K(A,B)\) and
\(\mu_{AC}=\mu_K(A,C)\).  Right: the resulting distance is an ultrametric,
so every triangle is isosceles with the two equal sides longest.}
\label{fig:dendrogram}
\end{figure}

\begin{remark}[What compactness buys]
Without \cref{thm:knot-sublevel-compact} one could still define \(\mu_K\)
and prove (i) and (ii); what would be missing is (iii).  And (iii) is the
whole point for \emph{connected} merge: it converts ``the configurations
belong to one connected component at every slightly larger level'' into the
same connectedness at the limiting level.  It makes \(d_K\) a metric on
connected ideal components.  It does not manufacture a path; that is the
separate issue of \cref{rem:path-vs-conn}.
\end{remark}

\subsection{Local path connectivity, and the equality of the two merge
functions}
\label{subsec:lpc}

The gap left by \cref{rem:path-vs-conn} is closed by a single estimate: a
straight-line interpolation between two \(C^1\)-close thick curves loses only
\(o(1)\) of thickness, and a homothety of ratio \(1+o(1)\) buys it back.  We
write \(\widetilde{Y}_\Lambda(K)\) for the set of centred constant-speed
\(C^{1,1}\) representatives of \cref{def:Y}, so that
\(Y_\Lambda(K)=\widetilde{Y}_\Lambda(K)/(\mathrm{SO}(3)\times\mathrm{O}(2))\), and we
use Federer's characterisation of reach \cite[Theorem 4.18]{Federer}: for a
\(C^{1,1}\) submanifold,
\begin{equation}
\label{eq:federer}
  \Thi=\inf_{p\ne q}\frac{|q-p|^{2}}{2\operatorname{dist}(q-p,\,T_p)} .
\end{equation}
For a regular curve \(z\) and \(\theta\ne\phi\) put
\(u=|z(\phi)-z(\theta)|\), \(D=|P^{\perp}_{\theta}(z(\phi)-z(\theta))|\) with
\(P^\perp_\theta\) the projection orthogonal to \(z'(\theta)\), and
\(\ell=d_{S^1}(\theta,\phi)\); thus \(\Thi(z)=\inf u^{2}/2D\).

\begin{lemma}[Uniform interpolation]
\label{lem:interpolation}
There are \(\delta_0>0\) and \(C<\infty\) depending only on \(\Lambda\) such
that, if \(x,y\in\widetilde{Y}_\Lambda(K)\) satisfy \(\|x-y\|_{C^1}<\delta<\delta_0\)
and \(z_\tau=(1-\tau)x+\tau y\), then for every \(\tau\in[0,1]\)
\[
  \Len(z_\tau)\le\max\{\Len(x),\Len(y)\}\le\Lambda,
  \qquad
  \Thi(z_\tau)\ge1-C\sqrt{\delta},
\]
and \(z_\tau\) is a \(C^{1,1}\) embedding of knot type \(K\), depending
continuously on \(\tau\) in \(C^1\).
\end{lemma}

\begin{proof}
Let \(v_x=\Len(x)/2\pi\) and \(v_y=\Len(y)/2\pi\); as both are of constant
speed, \(|v_x-v_y|\le|x'(\theta)-y'(\theta)|<\delta\) for every \(\theta\).
Constants \(C_1,C_2,\dots\) depend only on \(\Lambda\).

Length is a convex functional of the map, so
\(\Len(z_\tau)\le(1-\tau)\Len(x)+\tau\Len(y)\).

Pointwise \(|z_\tau'|\ge|x'|-\tau|x'-y'|\ge v_x-\delta\), and by
\cref{prop:unknot-2pi} \(v_x\ge1\); also \(|z_\tau'|\le v_x+\delta\).  Since
\(\Thi\ge1\) forces \(|\gamma''|\le v^{2}\) for a constant-speed curve of
speed \(v\),
\[
  |z_\tau''|\le\max\{v_x^{2},v_y^{2}\}\le(v_x+\delta)^{2},
  \qquad
  \kappa_{z_\tau}\le\frac{|z_\tau''|}{|z_\tau'|^{2}}
  \le\frac{(1+\delta)^{2}}{(1-\delta)^{2}}\le1+C_1\delta,
\]
the middle step because \(s\mapsto(s+\delta)^2/(s-\delta)^2\) decreases on
\(s>\delta\).  Write \(m=\inf|z_\tau'|\), \(V=\sup|z_\tau'|\), so
\(V/m\le1+C_2\delta\).  The estimate that makes the rest work is
\begin{equation}
\label{eq:refined}
  \bigl|(z_\tau-x)(\phi)-(z_\tau-x)(\theta)\bigr|
  =\tau\Bigl|\int_\theta^\phi(y'-x')\Bigr|\le\delta\,\ell ,
\end{equation}
which is \(O(\delta\ell)\), not merely \(O(\delta)\).

Fix \(\rho=\rho(\Lambda)>0\) so small that
\(m\ell(1-\tfrac12C_3\ell)\ge\tfrac12\ell\) for \(\ell\le\rho\), where
\(C_3=\sup|z''|/m\); then \(\tfrac12\ell\le u\le V\ell\) in that range.  We
bound \(2D/u^{2}\) in three ranges.

\emph{\(\ell\le\sqrt\delta\).}  As \(P^\perp_\theta z_\tau'(\theta)=0\),
\(D\le\int_\theta^\phi|P^\perp_\theta\int_\theta^\xi z_\tau''|
\le\tfrac12\ell^{2}(V^{2}\kappa_{z_\tau}+C_4\ell)\), the tangent turning by
\(O(\ell)\).  With \(u\ge m\ell(1-\tfrac12C_3\ell)\),
\[
  \frac{2D}{u^{2}}\le\Bigl(\frac{V}{m}\Bigr)^{2}\kappa_{z_\tau}(1+C_5\ell)
  \le1+C_6\delta+C_7\ell\le1+C_8\sqrt\delta .
\]

\emph{\(\sqrt\delta\le\ell\le\rho\).}  Here \(u_x\ge\tfrac12\ell\), so
\eqref{eq:refined} gives \(\delta\ell\le2\delta u_x\), whence
\(u\ge u_x(1-2\delta)\) and, the tangent direction rotating by at most
\(\delta/m\le2\delta\), \(D\le D_x+\delta\ell+2\delta u_x\le
D_x+C_9\delta u_x\).  Since \(\Thi(x)\ge1\), \eqref{eq:federer} gives
\(D_x\le u_x^{2}/2\), so
\[
  \frac{2D}{u^{2}}\le\frac{1+2C_9\delta/u_x}{(1-2\delta)^{2}}
  \le1+C_{10}\frac{\delta}{\ell}+C_{11}\delta\le1+C_{12}\sqrt\delta .
\]

\emph{\(\ell\ge\rho\).}  By \cref{thm:knot-sublevel-compact} the quantity
\(d_0=\inf\{|\gamma(\theta)-\gamma(\phi)|:\gamma\in\widetilde{Y}_\Lambda(K),\
d_{S^1}(\theta,\phi)\ge\rho\}\) is positive, so \(u_x\ge d_0\).  As above
\(u\ge u_x-\pi\delta\) and \(D\le u_x^{2}/2+C_{13}\delta\), whence
\(2D/u^{2}\le1+C_{14}d_0^{-2}\delta\).

Taking the supremum and inverting gives \(\Thi(z_\tau)\ge1-C\sqrt\delta\).
Positive thickness gives embeddedness, and a continuous family of embeddings
has constant knot type.
\end{proof}

\begin{remark}
The three ranges are forced: the first has error \(O(\ell)\) and the second
error \(O(\delta/\ell)\), so they meet at \(\ell\sim\sqrt\delta\).  Without
\eqref{eq:refined} the second error would be \(O(\delta/\ell^{2})\) and the
ranges would not overlap.  The exponent \(\tfrac12\) is certainly not
optimal; nothing below uses more than \(\Thi(z_\tau)\ge1-\eta(\delta)\) with
\(\eta(\delta)\to0\).
\end{remark}

\begin{theorem}[Local path connectivity after right relaxation]
\label{thm:lpc}
For every finite \(\Lambda\ge m_K\) and every \(\eps>0\) there is
\(\delta>0\) such that any two points of \(Y_\Lambda(K)\) at
\(C^1\)-distance less than \(\delta\) are joined by a path in
\(Y_{\Lambda+\eps}(K)\).
\end{theorem}

\begin{proof}
Choose \(\delta<\delta_0\) with \(C\sqrt\delta<\tfrac12\) and
\((1-C\sqrt\delta)^{-1}\Lambda\le\Lambda+\eps\), and put
\(\lambda=(1-C\sqrt\delta)^{-1}\).  Concatenate the homothety path from
\(x\) to \(\lambda x\), the scaled interpolation \(\tau\mapsto\lambda
z_\tau\), and the reverse homothety from \(\lambda y\) to \(y\).  A homothety
of ratio \(\mu\ge1\) multiplies both thickness and length by \(\mu\), so the
first and third have thickness at least \(1\) and length at most
\(\lambda\Lambda\le\Lambda+\eps\); on the middle one
\(\Thi(\lambda z_\tau)=\lambda\Thi(z_\tau)\ge\lambda(1-C\sqrt\delta)=1\) and
\(\Len(\lambda z_\tau)\le\lambda\Lambda\le\Lambda+\eps\) by
\cref{lem:interpolation}.  Reparametrising to constant speed is continuous
and changes neither quantity, so the concatenation descends to a path in
\(Y_{\Lambda+\eps}(K)\).
\end{proof}

\begin{corollary}
\label{cor:comp-to-path}
If \(x,y\) lie in one connected component of \(Y_\Lambda(K)\), then for every
\(\eps>0\) they lie in one path component of \(Y_{\Lambda+\eps}(K)\).
\end{corollary}

\begin{proof}
Declare \(p\sim q\) when they are joined by a path in \(Y_{\Lambda+\eps}(K)\).
By \cref{thm:lpc} every \(\delta\)-ball of \(Y_\Lambda(K)\) lies in one
class, so the classes are open, hence clopen, hence unions of connected
components.
\end{proof}

\begin{theorem}[Equality of the merge functions]
\label{thm:merge-equal}
\(\mu_K(x,y)=\mu_K^{\mathrm{path}}(x,y)\) for all \(x,y\in Y_\infty(K)\).
\end{theorem}

\begin{proof}
Put \(\Lambda=\mu_K(x,y)\), attained by \cref{thm:merge-ultrametric}(iii).
By \cref{cor:comp-to-path}, \(\mu_K^{\mathrm{path}}(x,y)\le\Lambda+\eps\) for
every \(\eps>0\).  The reverse inequality is formal.
\end{proof}

\begin{corollary}
\label{cor:ultrametric-path}
On \(\Conn(I(K))\) the function \(\mu_K^{\mathrm{path}}-m_K\) is a
finite-valued ultrametric; in particular the path merge separates distinct
connected components of the ideal stratum.
\end{corollary}

\begin{corollary}[Finiteness, right-relaxed form]
\label{cor:finiteness}
For \(\Lambda'<\Lambda\) the images of
\(\Conn(Y_{\Lambda'}(K))\to\Conn(Y_{\Lambda}(K))\) and of
\(\pi_0(Y_{\Lambda'}(K))\to\pi_0(Y_{\Lambda}(K))\) are finite.
\end{corollary}

\begin{proof}
Apply \cref{thm:lpc} at level \(\Lambda'\) with \(\eps=\Lambda-\Lambda'\) to
get \(\delta\); by \cref{thm:knot-sublevel-compact} the space
\(Y_{\Lambda'}(K)\) is covered by finitely many \(\delta/2\)-balls, and each
lies in a single path component of \(Y_\Lambda(K)\).
\end{proof}

\begin{corollary}[Uniform finite-level merging of a fixed sublevel]
\label{cor:uniform-finite-merging}
For every knot type \(K\) and every finite \(\Lambda\ge m_K\), there is a
finite \(L\ge\Lambda\) such that \(Y_\Lambda(K)\) is contained in a single
path component of \(Y_L(K)\).  Equivalently,
\[
  \#\im\bigl[\pi_0(Y_\Lambda(K))\longrightarrow\pi_0(Y_L(K))\bigr]=1.
\]
In particular, all configurations in the ideal stratum \(I(K)=Y_{m_K}(K)\)
can be joined to one another at a common finite level, without any
finiteness assumption on the components of \(I(K)\).
\end{corollary}

\begin{proof}
Fix \(\eps>0\).  By \cref{cor:finiteness}, only finitely many path
components \(P_1,\ldots,P_N\) of \(Y_{\Lambda+\eps}(K)\) meet
\(Y_\Lambda(K)\).  Choose \(x_i\in P_i\cap Y_\Lambda(K)\).
The proof of \cref{lem:finiteness-merge} supplies, for each \(i\), a path
from \(x_i\) to \(x_1\) in some finite sublevel \(Y_{L_i}(K)\).
Set
\[
  L=\max\{\Lambda+\eps,L_1,\ldots,L_N\}<\infty.
\]
Every \(x\in Y_\Lambda(K)\) belongs to some \(P_i\), so it can be joined
to \(x_i\) inside \(Y_{\Lambda+\eps}(K)\), and then to \(x_1\) inside
\(Y_L(K)\).  Thus all of \(Y_\Lambda(K)\) lies in one path component of
\(Y_L(K)\).  Taking \(\Lambda=m_K\) gives the assertion for the ideal
stratum.
\end{proof}

The bound in \cref{cor:uniform-finite-merging} concerns the image of a
fixed sublevel: it does not assert that the whole space \(Y_L(K)\) is path
connected, or that all sufficiently high sublevels are path connected.
The argument gives a finite common budget, rather than an effective
numerical bound or attainment of the least such path budget.

\begin{corollary}
\label{cor:admissible-agree}
The threshold at which two configurations become joinable by an admissible
isotopy in the regularity class of \cite{OzawaSwept} equals
\(\mu_K=\mu_K^{\mathrm{path}}\).
\end{corollary}

\begin{proof}
By \cref{cor:comp-to-path} and the proof of \cref{cor:finiteness} the two
points are joined by a finite chain of \(\delta\)-steps, each realised by the
explicit path of \cref{thm:lpc}, which is a concatenation of homotheties and
straight-line interpolations and hence jointly \(C^{1,1}\) with tangent
bounded away from zero.  Concatenating gives an admissible isotopy at level
\(\Lambda+\eps\) for every \(\eps>0\).
\end{proof}

\subsubsection*{Where the remaining obstruction sits}

\begin{proposition}[Local path connectivity off the shell]
\label{prop:lpc-off-shell}
Let \(x\in Y_\Lambda(K)\) with \(\Len(x)<\Lambda\).  Then some
\(C^1\)-neighbourhood of \(x\) in \(Y_\Lambda(K)\) consists of points joined
to \(x\) by a path \emph{inside} \(Y_\Lambda(K)\).
\end{proposition}

\begin{proof}
Write \(\Len(x)=\Lambda-\eta\).  For \(\|x-y\|_{C^1}<\delta\) we have
\(\Len(y)\le\Len(x)+2\pi\delta\), so the path of \cref{thm:lpc} has length at
most \(\lambda(\Lambda-\eta+2\pi\delta)
=\Lambda-\eta+2\pi\delta+C\sqrt\delta\,\Lambda+O(\delta)\), which is at most
\(\Lambda\) once \(2\pi\delta+C\Lambda\sqrt\delta\le\eta\).
\end{proof}

\begin{corollary}
\label{cor:obstruction-on-shell}
The obstruction to local path connectivity of \(Y_\Lambda(K)\) itself is
confined to the exact-length shell \(\{\Len=\Lambda\}\), equivalently to the
configurations with \(\Rop=\Lambda\).  If that shell is locally path
connected then \(\pi_0(Y_\Lambda(K))=\Conn(Y_\Lambda(K))\) is finite and the
infimum defining \(\mu_K^{\mathrm{path}}\) is attained.
\end{corollary}

\begin{remark}[The localisation is precise but not encouraging]
\label{rem:attainment-open}
By \cref{prop:fixed-length-retract} the space \(Y_\Lambda(K)\) strongly
deformation retracts onto exactly that shell, so all of its topology lives
where \cref{prop:lpc-off-shell} says nothing.  What the localisation gives is
a clean statement of what remains, and the knowledge that the interpolation
estimate is not the missing ingredient.  The difficulty is the regularity of
the constraint: by \eqref{eq:three-point} thickness is an infimum of
circumradii, hence upper semicontinuous and genuinely discontinuous --- a
circle of radius \(2\) is a \(C^1\)-limit of curves carrying small
high-frequency wiggles of curvature \(1\), along which \(\Thi\equiv1\) --- so
\(\Rop\) is only lower semicontinuous, and the sublevel set of a lower
semicontinuous function need not be locally path connected where the
constraint is active.  Closing \cref{cor:obstruction-on-shell} therefore
needs the structure of ropelength-critical configurations, for which the
balance criterion of \cite{CFKSW} and the characterisation of
\cite{SchurichtvdMosel} are the natural tools.
\end{remark}

\subsection{Relation to deformation cost}\label{subsec:swept-area}
Merge scales measure the budget at which a connection becomes available.
The independent cost theory of \cite{OzawaSwept} measures the swept area of
an admissible connection. The interpolation paths in \cref{thm:lpc} give
admissible endpoint connections after every positive budget relaxation.
Thus the infimal thresholds agree by \cref{thm:merge-equal}; this assertion
neither proves attainment by a path at the threshold nor convergence of
minimal swept-area costs. Metric separation and calibration estimates belong
to the cost theory and are not repeated here.

\begin{question}[Finiteness of connected components]\label{q:finite-components}
Is \(\Conn(Y_\Lambda(K))\) finite for every knot type \(K\) and every finite
\(\Lambda\)?  Compactness alone does not give this --- a compact space can
have a Cantor set of components.  \Cref{cor:finiteness} answers the
right-relaxed form: for \(\Lambda'<\Lambda\) only finitely many components of
\(Y_{\Lambda'}(K)\) survive to level \(\Lambda\), in both the connected and
the path sense, so the filtration \(\Lambda\mapsto\pi_0(Y_\Lambda(K))\) is
\(q\)-tame and its persistence diagram is well defined.  Finiteness at a
single level remains open, and by \cref{cor:obstruction-on-shell} the
obstruction is confined to the exact-length shell.  A positive answer would
say that a rope of bounded length admits only finitely many essentially
different thick placements.
\end{question}

\subsection{Connected and path untying thresholds}
\label{subsec:gordian}

Fix a connected component \(A_0\subset I(K)\) that is to serve as a tight
target; for the unknot, \(A_0=I(U)\) is the single round unit circle by
\cref{prop:unknot-2pi}.  Let \(x\in Y_\infty(K)\) and put
\(L=\Len(x)\).

\begin{definition}[Two untying thresholds]\label{def:untying}
Define
\[
 u_{\mathrm c}(x)=\mu_K(x,A_0),\qquad
 u_{\mathrm p}(x)=\mu_K^{\mathrm{path}}(x,A_0),
\]
and the corresponding excesses
\[
 e_{\mathrm c}(x)=u_{\mathrm c}(x)-L,\qquad
 e_{\mathrm p}(x)=u_{\mathrm p}(x)-L.
\]
We call \(u_{\mathrm c}\) the \emph{connected untying threshold} and
\(u_{\mathrm p}\) the \emph{path-budget untying threshold}.
We say that \(x\) is \emph{Gordian relative to \(A_0\) at its own length}
if there is no path from \(x\) to \(A_0\) in \(Y_L(K)\).
\end{definition}

The Gordian definition at the starting level has an exact fixed-length
physical interpretation.  Indeed, a path from \(x\) to \(A_0\) in
\(Y_L(K)\) can be pushed by \cref{prop:fixed-length-retract} to an
exact-length path from \(x\) to the dilation of \(A_0\) to length \(L\);
conversely the dilation path back to \(A_0\) remains in \(Y_L(K)\).  For
levels larger than \(L\), however, \(u_{\mathrm p}\) should be read as a
min--max \emph{budget} for sublevel paths, not literally as the length of
the original physical rope.

\begin{theorem}[The untying threshold]
\label{thm:untying-dichotomy}
With the notation above:
\begin{enumerate}[label=\textup{(\roman*)},leftmargin=2.7em]
\item
\[
  L\le u_{\mathrm c}(x)=u_{\mathrm p}(x)<\infty,
  \qquad 0\le e_{\mathrm c}(x)=e_{\mathrm p}(x)<\infty.
\]
\item The threshold is attained as a connected merge level: \(x\) and
\(A_0\) lie in one connected component of \(Y_{u_{\mathrm c}(x)}(K)\).
\item Every path from \(x\) to \(A_0\) has peak length at least
\(u_{\mathrm p}(x)\), while for every \(\eps>0\) there is such a path in
\(Y_{u_{\mathrm p}(x)+\eps}(K)\).
\item If \(e_{\mathrm c}(x)>0\), then \(x\) is Gordian at length \(L\),
and every sublevel path from \(x\) to \(A_0\) must reach length at least
\(L+e_{\mathrm c}(x)\).
\item If \(e_{\mathrm c}(x)=0\), then \(x\) and \(A_0\) lie in one connected
component of \(Y_L(K)\), and they are joined by a path in
\(Y_{L+\eps}(K)\) for every \(\eps>0\).  The threshold alone does
\emph{not} decide whether \(x\) is Gordian at length \(L\): this is exactly
the question whether the path threshold is attained.
\end{enumerate}
\end{theorem}

\begin{proof}
The lower bound \(L\le u_{\mathrm c}\) is forced because \(x\notin
Y_\Lambda(K)\) for \(\Lambda<L\).  The equality
\(u_{\mathrm c}=u_{\mathrm p}\) is \cref{thm:merge-equal}, and finiteness is
\cref{lem:finiteness-merge}.  Attainment of \(u_{\mathrm c}\) is
\cref{thm:merge-ultrametric}(iii).  The assertions of (iii) are the
definition of \(u_{\mathrm p}\) as an infimum of path-admissible levels.  If
\(u_{\mathrm c}>L\), then no path can exist in \(Y_L(K)\), since a path would
place \(x\) and \(A_0\) in one connected component there.  If
\(u_{\mathrm c}=L\), (ii) and (iii) give the two assertions of (v), and a
path in \(Y_L(K)\) exists precisely when the infimum defining
\(u_{\mathrm p}(x)=L\) is attained.
\end{proof}

\begin{figure}[t]
\centering
\begin{tikzpicture}[x=1cm,y=1cm]
\draw[->,>=Stealth,line width=0.8pt] (0,0) -- (7.4,0)
  node[right,font=\scriptsize] {deformation parameter};
\draw[->,>=Stealth,line width=0.8pt] (0,0) -- (0,3.35)
  node[above,font=\scriptsize] {length};
% threshold u = u_c = u_p and starting level L
\draw[ropegreen,dashed,line width=0.75pt] (0,2.1) -- (7.15,2.1);
\node[ropegreen,font=\scriptsize,left] at (-0.05,2.1) {$u_{\rm c}=u_{\rm p}$};
\draw[ropegrey,dashed,line width=0.6pt] (0,1.35) -- (7.15,1.35);
\node[font=\scriptsize,left] at (-0.05,1.35) {$L$};
% two paths from x to A_0 whose peaks approach the threshold from above
\draw[ropeblue!55,line width=1.0pt]
  (0.35,1.35) .. controls (1.3,1.35) and (1.8,2.95) .. (2.8,2.95)
             .. controls (3.8,2.95) and (3.5,1.25) .. (4.7,1.25)
             .. controls (5.7,1.25) and (5.7,0.75) .. (6.8,0.75);
\draw[ropeblue,line width=1.2pt]
  (0.35,1.35) .. controls (1.4,1.35) and (2.1,2.25) .. (3.0,2.25)
             .. controls (3.9,2.25) and (3.7,1.2) .. (4.8,1.15)
             .. controls (5.8,1.1) and (5.8,0.75) .. (6.8,0.75);
\filldraw[ropered] (0.35,1.35) circle (1.8pt);
\node[ropered,font=\scriptsize,above left] at (0.35,1.38) {$x$};
\filldraw[ropegreen] (6.8,0.75) circle (1.8pt);
\node[ropegreen,font=\scriptsize,right] at (6.85,0.75) {$A_0$};
\draw[<->,>=Stealth,ropered,line width=0.75pt] (7.0,1.35) -- (7.0,2.1);
\node[ropered,font=\scriptsize,right] at (7.05,1.72) {$e$};
\end{tikzpicture}
\caption{The untying threshold.  By \cref{thm:merge-equal} the connected
threshold \(u_{\rm c}\) and the path threshold \(u_{\rm p}\) coincide.  Their
common value is attained as a connected merge level, and every path from
\(x\) to \(A_0\) peaks at or above it; paths whose peaks come arbitrarily
close to it exist (the two drawn profiles), but a path with peak exactly
\(u_{\rm p}\) is not known to exist.  Only at the starting level \(L\) does
path connectivity directly coincide with fixed-length physical isotopy.  The
drawn profiles are only illustrative.}
\label{fig:untying}
\end{figure}

\begin{remark}[The state of the art, model by model]\label{rem:state-of-art}
The theorem above is deliberately phrased so that different physical models
are not mixed.
\begin{itemize}[leftmargin=1.4em]
\item In fixed-length thick \emph{link} isotopy, Coward--Hass
\cite{CowardHass} prove a Gordian split link.  Kusner--Kusner
\cite{KusnerKusner} give a Gordian pair under a total-length constraint, and
Ayala--Hass \cite{AyalaHass} construct infinite families of Gordian unlinks.
Bauermeister \cite{Bauermeister} proves further split-link obstructions in
the Gehring ropelength setting, where length trading is allowed.
\item For a single \emph{unknot}, Ayala \cite{AyalaGeometric} proves Gordian
unknots in a thinner model \(\mathcal U_1\) of parametrised thickness.  In the
same notation the standard thick model is \(\mathcal U_2\).  Thus this
recent theorem is close to, but not identical with, the standard
\(Y_L(U)\) problem studied here.  Ayala singles out the saturated case
\(\tau=2\) as a qualitatively different regime, in which the geometric
slack used in the thin constructions is no longer available, and suggests
that Morse-theoretic or critical-point methods may be required there.  The
sublevel formulation of the present paper is one attempt to supply a
framework of that kind.
\item In the standard unit-thickness model of this paper, a rigorous example
with \(x\) path-separated from the round target inside \(Y_{\Len(x)}(U)\)
would settle \cref{prob:gordian-knot}.  The stronger inequality
\(u_{\mathrm c}(x)>\Len(x)\) would additionally provide an attained
connected lower-bound certificate.
\end{itemize}
The moral is that ``Gordian'' should always be accompanied by the precise
constraints on tube radius, curvature, component lengths, and allowed
length trading.
\end{remark}

\begin{figure}[t]
\centering
\begin{tikzpicture}[scale=0.95]
% back half of the ring
\draw[line width=6pt,ropesand] (0,0) ++(200:1.05) arc (200:-20:1.05);
\draw[line width=1pt,ropeblue] (0,0) ++(200:1.05) arc (200:-20:1.05);
% threaded strand, drawn on top of the back half
\draw[line width=6pt,ropesand]
  (-2.6,0.02) .. controls (-1.3,0.5) and (-0.3,0.6) .. (0.0,0.0)
              .. controls (0.3,-0.6) and (1.3,-0.5) .. (2.6,-0.02);
\draw[line width=1pt,ropered]
  (-2.6,0.02) .. controls (-1.3,0.5) and (-0.3,0.6) .. (0.0,0.0)
              .. controls (0.3,-0.6) and (1.3,-0.5) .. (2.6,-0.02);
% front half of the ring, drawn on top of the strand
\draw[line width=6pt,ropesand] (0,0) ++(-20:1.05) arc (-20:-160:1.05);
\draw[line width=1pt,ropeblue] (0,0) ++(-20:1.05) arc (-20:-160:1.05);
\node[font=\scriptsize,align=center] at (0,-1.7)
  {a strand of a short component threaded\\ through a tight ring};
\draw[->,>=Stealth,ropegrey,line width=0.9pt] (2.9,0.0) -- (3.7,0.0);
\node[font=\scriptsize,align=left,anchor=west] at (4.0,0)
  {to withdraw, the strand must first\\ pass through the ring, and the ring\\
   is too tight unless extra rope is spent};
\end{tikzpicture}
\caption{The mechanism behind Gordian phenomena, schematically.  A short
component cannot be withdrawn without opening the ring, and opening the ring
costs length.  Making this into a proof is exactly the content of
\cite{CowardHass}; the picture is a caricature, not their example.}
\label{fig:gordian-mechanism}
\end{figure}

\subsection{Mirror symmetry}

Ropelength is invariant under all isometries of \(\R^3\), including
orientation-reversing ones, and this gives the filtration an extra symmetry
that is invisible in the classical theory.

\begin{proposition}[Mirror involution]\label{prop:mirror-involution}
Fix a reflection \(r\) of \(\R^3\).  Then
\(\tau_\Lambda([\gamma])=[r\circ\gamma]\) is a well-defined homeomorphism
\[
  \tau_\Lambda:Y_\Lambda(K)\longrightarrow Y_\Lambda(\overline K)
\]
for every \(\Lambda\), compatible with the inclusions of sublevels.  If
\(K\) is amphichiral, \(\tau\) is an involution of the whole filtration of
\(K\) and restricts to an involution of \(I(K)\), so that
\(\{Y_\Lambda(K)\}_\Lambda\) is a filtered \(\Z/2\)-space.
\end{proposition}

\begin{proof}
Reflections preserve length and thickness, and normalise
\(\Isom^+(\R^3)\), so the formula is independent of the representative and
of the choice of \(r\) up to conjugation.  It is a homeomorphism because
\(r^2=\id\).
\end{proof}

\begin{definition}[Mirror merge scale]\label{def:mirror-merge}
For amphichiral \(K\), choose a connected component \(A\) of \(I(K)\), and put
\[
  \mu^{\mathrm{mir}}_K(A)=\mu_K(A,\tau A),\qquad
  d^{\mathrm{mir}}_K(A)=\mu^{\mathrm{mir}}_K(A)-m_K.
\]
\end{definition}

\begin{corollary}[Mirror barrier criterion]
\label{cor:mirror-barrier-criterion}
If \(A\ne\tau A\), then \(0<d^{\mathrm{mir}}_K(A)<\infty\), and the barrier
is attained.
\end{corollary}

\begin{proof}
\Cref{thm:merge-ultrametric} applied to \(A\) and \(\tau A\).
\end{proof}

\Cref{cor:mirror-barrier-criterion} separates the soft part of the problem
from the hard part.  Compactness converts mirror separation at the ideal
level into a positive, attained barrier automatically; what nobody can
currently do is exhibit an amphichiral knot type with a mirror-nonfixed
ideal component.

\begin{figure}[t]
\centering
\begin{tikzpicture}[scale=1.0]
\draw[fill=ropesand,draw=ropegrey,line width=0.7pt,rounded corners=12pt]
  (0,0) rectangle (9.2,2.6);
\node[font=\scriptsize,below right] at (0.1,-0.05) {$Y_\Lambda(K)$, $K$ amphichiral};
\draw[ropegrey,dashed,line width=0.8pt] (4.6,-0.1) -- (4.6,2.7);
\node[font=\scriptsize,above] at (4.6,2.7) {$\tau$};
\draw[ropeblue,line width=1pt,fill=ropelight] (2.3,1.3) ellipse (1.5 and 0.85);
\draw[ropeblue,line width=1pt,fill=ropelight] (6.9,1.3) ellipse (1.5 and 0.85);
\node[ropeblue,font=\scriptsize] at (2.3,1.3) {$A$};
\node[ropeblue,font=\scriptsize] at (6.9,1.3) {$\tau A$};
\draw[<->,>=Stealth,ropered,line width=0.9pt] (3.85,1.15) to[bend right=25] (5.35,1.15);
\node[ropered,font=\scriptsize,fill=ropesand,inner sep=1.5pt] at (4.6,0.45) {$d^{\mathrm{mir}}_K(A)>0$};
\end{tikzpicture}
\caption{If an ideal component is not preserved by the mirror involution,
\cref{cor:mirror-barrier-criterion} produces a strictly positive attained
amount of extra rope needed to deform a tight configuration into its own
mirror image.  For the figure-eight knot this is
\cref{prob:figure-eight-mirror}.}
\label{fig:mirror}
\end{figure}

\subsection{The figure-eight knot}

The figure-eight knot \(4_1\) is the natural first test: it is amphichiral,
it is hyperbolic, so \cref{thm:hatcher-models}(ii) applies and its terminal
knot space is understood, and its ropelength is well studied numerically
(\cref{fig:numberline}).

\begin{problem}[Figure-eight mirror separation]
\label{prob:figure-eight-mirror}
Does there exist a connected component \(A\subset I(4_1)\) with
\(A\ne\tau A\)?
\end{problem}

A positive answer would produce, via
\cref{cor:mirror-barrier-criterion}, the first continuum ropelength barrier
for a knot.  We stress that \cref{prob:figure-eight-mirror} is a question
about the \emph{connected topology of the whole minimiser locus}, and is
stronger than the statement that some numerically computed tight
configuration fails to be congruent to its reflection.

The lattice merge trees in \cite{OzawaDiscrete} provide separate discrete
benchmarks. Lattice edge number and Euclidean ropelength are different
resources; no continuum barrier follows without controlled comparison maps.
The construction and certification of those trees belong to that paper.

\subsection{Degree zero as the first layer of one theory}

The reader should notice that
\cref{thm:merge-ultrametric,prop:width-properties} share the same
non-Archimedean min--max pattern, although they are not literally the same
homotopy-theoretic construction.  For connected merge, concatenating two
mergers gives
\(\mu(x,z)\le\max\{\mu(x,y),\mu(y,z)\}\); in positive degree the pinch or
concatenation operation gives
\(\wid(\alpha\beta)\le\max\{\wid\alpha,\wid\beta\}\).  The extra feature of
the \emph{connected} degree-zero theory is attainment, obtained from
compactness of \(Y_\Lambda(K)\).  Path merge does not presently share this
advantage.  Whether an analogous compactness or attainment principle exists
for families in positive degree is \cref{prob:attainment}.

\section{Three bridges to the Hatcher--Budney picture}
\label{sec:bridges}

\subsection{Closed knots in the round three-sphere}

Give \(S^3\) the round metric and let \(\mathscr K_{K,\Lambda}(S^3)\) be the
spherical ropelength sublevel of \(\mathscr K_K(S^3)\).

\begin{theorem}[Persistent recovery of Hatcher's knot space]
\label{thm:hatcher-recovery}
For every \(i\ge0\),
\(\colim_\Lambda\pi_i(\mathscr K_{K,\Lambda}(S^3))\cong
\pi_i(\mathscr K_K(S^3))\), and for every cofinal sequence of levels the
canonical map
\(\hocolim_\Lambda\mathscr K_{K,\Lambda}(S^3)\to\mathscr K_K(S^3)\)
is a weak homotopy equivalence.
\end{theorem}

\begin{proof}
\Cref{cor:quotient-exhaustion,cor:telescope} for the round \(S^3\).
\end{proof}

Combining with \cref{thm:hatcher-models} gives, for a torus or hyperbolic
knot,
\[
  \colim_\Lambda\pi_i(\mathscr K_{K,\Lambda}(S^3))\cong\pi_i(SO(4)/H_K),
\]
where \(H_K\) is the stabiliser of a symmetric placement.  By
\cref{thm:vanishing-width} the stable homotopy groups are already hit surjectively at the finite
level \(\Rop_{S^3}(J_0)\).  This is only a statement about classes that
survive to the terminal knot space: transient classes may still be born and
die at other finite levels, so the filtered positive-degree theory is not
thereby exhausted.

\subsection{Long knots by compactification}
\label{subsec:longknots}

A long knot has infinite length, so Euclidean ropelength is not defined for
it.  One-point compactification repairs this.  Let
\(c:\calK\to\Emb^\infty(S^1,S^3)\) send a long knot to its closure in
\(S^3=\R^3\cup\{\infty\}\), and put
\[
  \Rop_c(f)=\Rop_{S^3}(c(f)),
  \qquad
  \calK_{f,\Lambda}=\{h\in\calK_f:\Rop_c(h)\le\Lambda\}.
\]
The individual sublevels depend on the compactification, but their limit
does not.

\begin{theorem}[Long-knot exhaustion]\label{thm:longknot-exhaustion}
For every long knot \(f\) and every \(i\ge0\),
\(\colim_\Lambda\pi_i(\calK_{f,\Lambda},f)\cong\pi_i(\calK_f,f)\).
\end{theorem}

\begin{proof}
The image under \(c\) of a compact family of long knots is a compact family
of smooth knots in \(S^3\); apply
\cref{thm:compact-family-bound} and repeat the proof of
\cref{thm:homotopical-exhaustion}.
\end{proof}

\Cref{cor:gramain-width,cor:transient-higher} are the two computations this
model supports: the Gramain class costs nothing, and every degree-\(\ge2\)
class is transient.  Thus the Gramain direction has zero excess, while every
degree-\(\ge2\) class visible at a finite level is transient.  No uniform
upper bound on the birth or death levels of all such transient classes is
asserted.

\subsection{The pole-marked complement bundle}

The round-metric filtration is the shortest route to Hatcher's homotopy
groups, but it discards the Euclidean picture entirely.  A moving
stereographic pole gives a second bridge which keeps the Euclidean geometry
and, as a bonus, remembers the knot complement.

For \(J\in\mathscr K_K(S^3)\) and \(p\in S^3\setminus J\), let
\(\sigma_p:S^3\setminus\{p\}\to\R^3\) be a stereographic projection from
\(p\).  Changing the target frame changes \(\sigma_p\) by a Euclidean
similarity, so
\[
  \calR(J,p)=\Rop_{\R^3}\bigl(\sigma_p(J)\bigr)
\]
is well defined.  Put
\(\calP(K)=\{(J,p):J\in\mathscr K_K(S^3),\ p\in S^3\setminus J\}\) and
\(\calP_\Lambda(K)=\{\calR\le\Lambda\}\).

\begin{theorem}[Pole bundle]\label{thm:pole-bundle}
The forgetful map \(\pi:\calP(K)\to\mathscr K_K(S^3)\), \((J,p)\mapsto J\),
is a locally trivial fibre bundle with fibre the knot complement
\(S^3\setminus K\); in particular it is a Serre fibration.
\end{theorem}

\begin{proof}
Fix \(J_0\).  By the isotopy extension theorem there is, on a small
normal-graph neighbourhood \(U\) of \(J_0\), a continuous family of ambient
diffeomorphisms \(h_J\) with \(h_J(J_0)=J\).  Then
\((J,x)\mapsto(J,h_J(x))\) trivialises \(\pi\) over \(U\).
\end{proof}

\begin{theorem}[Pole-space exhaustion]\label{thm:pole-exhaustion}
Every compact subset of \(\calP(K)\) lies in \(\calP_\Lambda(K)\) for some
finite \(\Lambda\).  Hence
\(\colim_\Lambda\pi_i(\calP_\Lambda(K))\cong\pi_i(\calP(K))\) for all
\(i\ge0\), and likewise for homology and mapping telescopes.
\end{theorem}

\begin{proof}
On a compact subset the spherical distance from the pole to the knot is
bounded below, so on finitely many local trivialisations the projections
\(\sigma_p\) and their first two derivatives are uniformly controlled along
the family.  \Cref{lem:local-tubular-stability} then bounds the Euclidean
thickness below and the length above, so \(\calR\) is locally bounded; now
argue as in \cref{thm:homotopical-exhaustion}.
\end{proof}

\begin{figure}[t]
\centering
\begin{tikzpicture}[scale=0.87]
\draw[ropegrey,line width=0.8pt] (0,0) circle (1.5);
\draw[ropegrey,line width=0.5pt] (0,0) ellipse (1.5 and 0.5);
\begin{scope}[shift={(-0.1,0.2)}]
\draw[line width=1.2pt,ropeblue,line join=round] plot[domain=0:360,samples=181,variable=\t] ({0.27*(sin(\t)+2*sin(2*\t))},{0.27*(cos(\t)-2*cos(2*\t))}) -- cycle;
\draw[line width=4pt,white] plot[domain=-0.5:31.6,samples=15,variable=\t] ({0.27*(sin(\t)+2*sin(2*\t))},{0.27*(cos(\t)-2*cos(2*\t))});
\draw[line width=1.2pt,ropeblue,line cap=round] plot[domain=-3.5:34.6,samples=17,variable=\t] ({0.27*(sin(\t)+2*sin(2*\t))},{0.27*(cos(\t)-2*cos(2*\t))});
\draw[line width=4pt,white] plot[domain=239.5:271.6,samples=15,variable=\t] ({0.27*(sin(\t)+2*sin(2*\t))},{0.27*(cos(\t)-2*cos(2*\t))});
\draw[line width=1.2pt,ropeblue,line cap=round] plot[domain=236.5:274.6,samples=17,variable=\t] ({0.27*(sin(\t)+2*sin(2*\t))},{0.27*(cos(\t)-2*cos(2*\t))});
\draw[line width=4pt,white] plot[domain=119.5:151.6,samples=15,variable=\t] ({0.27*(sin(\t)+2*sin(2*\t))},{0.27*(cos(\t)-2*cos(2*\t))});
\draw[line width=1.2pt,ropeblue,line cap=round] plot[domain=116.5:154.6,samples=17,variable=\t] ({0.27*(sin(\t)+2*sin(2*\t))},{0.27*(cos(\t)-2*cos(2*\t))});
\end{scope}
\filldraw[ropered] (0.75,-1.05) circle (1.8pt);
\node[ropered,font=\scriptsize,right] at (0.8,-1.05) {$p$};
\node[ropeblue,font=\scriptsize] at (-1.15,0.95) {$J$};
\node[font=\scriptsize,below] at (0,-1.75) {$S^3$};
\draw[->,>=Stealth,line width=0.9pt] (1.85,0) -- (2.85,0);
\node[font=\scriptsize,above] at (2.35,0.05) {$\sigma_p$};
\begin{scope}[shift={(4.9,0)}]
  \draw[ropegrey,line width=0.5pt] (-1.35,-1.1) rectangle (1.35,1.1);
  \begin{scope}[shift={(0,0.1)},xscale=1.12,yscale=0.95]
\draw[line width=1.2pt,ropeblue,line join=round] plot[domain=0:360,samples=181,variable=\t] ({0.27*(sin(\t)+2*sin(2*\t))},{0.27*(cos(\t)-2*cos(2*\t))}) -- cycle;
\draw[line width=4pt,white] plot[domain=-0.5:31.6,samples=15,variable=\t] ({0.27*(sin(\t)+2*sin(2*\t))},{0.27*(cos(\t)-2*cos(2*\t))});
\draw[line width=1.2pt,ropeblue,line cap=round] plot[domain=-3.5:34.6,samples=17,variable=\t] ({0.27*(sin(\t)+2*sin(2*\t))},{0.27*(cos(\t)-2*cos(2*\t))});
\draw[line width=4pt,white] plot[domain=239.5:271.6,samples=15,variable=\t] ({0.27*(sin(\t)+2*sin(2*\t))},{0.27*(cos(\t)-2*cos(2*\t))});
\draw[line width=1.2pt,ropeblue,line cap=round] plot[domain=236.5:274.6,samples=17,variable=\t] ({0.27*(sin(\t)+2*sin(2*\t))},{0.27*(cos(\t)-2*cos(2*\t))});
\draw[line width=4pt,white] plot[domain=119.5:151.6,samples=15,variable=\t] ({0.27*(sin(\t)+2*sin(2*\t))},{0.27*(cos(\t)-2*cos(2*\t))});
\draw[line width=1.2pt,ropeblue,line cap=round] plot[domain=116.5:154.6,samples=17,variable=\t] ({0.27*(sin(\t)+2*sin(2*\t))},{0.27*(cos(\t)-2*cos(2*\t))});
  \end{scope}
  \node[font=\scriptsize,below] at (0,-1.2) {$\sigma_p(J)\subset\R^3$};
\end{scope}
\begin{scope}[shift={(9.0,0)}]
  \draw[rounded corners=10pt,fill=ropesand,draw=ropegrey,line width=0.7pt]
    (-1.5,0.55) rectangle (1.5,1.35);
  \node[font=\scriptsize] at (0,0.95) {$\calP(K)$};
  \draw[->,>=Stealth,line width=0.8pt] (0,0.5) -- (0,-0.25);
  \node[font=\scriptsize,right] at (0.08,0.15) {$\pi$};
  \draw[rounded corners=10pt,fill=ropesand,draw=ropegrey,line width=0.7pt]
    (-1.5,-1.1) rectangle (1.5,-0.3);
  \node[font=\scriptsize] at (0,-0.7) {$\mathscr K_K(S^3)$};
  \node[font=\scriptsize,right,align=left] at (1.6,0.95)
    {fibre\\ $S^3\setminus K$};
\end{scope}
\end{tikzpicture}
\caption{The pole-marked space.  A point is a knot in \(S^3\) together with
a projection pole, and stereographic projection from that pole makes
Euclidean ropelength available.  Forgetting the pole is a bundle with fibre
the knot complement (\cref{thm:pole-bundle}), so its filtration sees the
knot space, the complement, and the geometry of thick representatives at the
same time.}
\label{fig:pole}
\end{figure}

\begin{remark}
The fibre is the reason to care.  In \(\calP(K)\) the topology of the
complement --- the source of the JSJ decomposition, of the symmetry group,
and of the fundamental group used in \cref{thm:hatcher-models} --- is
present as literal geometry, and the filtration parameter is a Euclidean
quantity.  This makes \(\calP(K)\) the natural place to look for the
satellite phenomena of \cref{prob:satellite-width}.
\end{remark}

\section{Relation to diagrammatic models}\label{sec:program}
The projection-framed construction of \cite{OzawaFinite} retains regular
projection-fibre components and liftable Reidemeister transitions. It treats
component reconstruction and the analytic smoothing needed for that purpose.
Higher reconstruction requires both fibre topology and relative genericity
for families; the ordinary graph alone does not encode these data.
The present paper concerns the underlying geometric filtration. Its results
are formulated independently of a diagrammatic reconstruction conjecture.

\section{Finite factorization in a marked weak model}
\label{sec:marked-finite}
This section uses a specified model rather than silently transporting claims
through a symmetry quotient. Let $S^1=\mathbb R/\mathbb Z$ and set
\[
 X_a=\{\gamma:S^1\hookrightarrow\mathbb R^3:\gamma\in C^{1,1},\
 |\gamma'|=1,\ \int\gamma=0,\ \Thi(\gamma)\ge a^{-1},\ [\gamma]=K\}.
\]
The topology is $C^1$. Phase, orientation and ambient rotation remain marked.
For a regular curve $z$, let $N(z)$ denote rescaling to length one,
constant-speed reparametrization starting at parameter zero, and subtraction
of the parameter centroid. More explicitly, if $\ell=\int_0^1|z'|$ and
$\theta(s)=\ell^{-1}\int_0^s|z'|$, first use $\ell^{-1}z\circ\theta^{-1}$.
This normalization is continuous in $C^1$ on regular curves. No continuity
in the strong second-derivative norm is asserted.

\begin{lemma}[Compactness and short chords]\label{lem:marked-compact}
$X_a$ is compact. If $\gamma\in X_a$, $h$ is circular parameter distance,
and $h\le 1/a$, then
$|\gamma(s+h)-\gamma(s)|\ge h-ah^2/2\ge h/2$.
For $0<h_0\le\min(1/4,1/a)$, chords with circular parameter distance at
least $h_0$ have length at least $h_0/2$.
\end{lemma}
\begin{proof}
The curvature bound gives $\operatorname{Lip}(\gamma')\le a$, and integration
of $\gamma'(s+h)-\gamma'(s)$ gives the short-chord estimate. A shortest
chord on the compact set of pairs with separation at least $h_0$ is either
on its boundary, where that estimate applies, or is doubly critical. In the
latter case its length is at least $2/a$ by the thickness formula.
The centroid condition bounds position; Arzel\`a--Ascoli now gives $C^1$
subsequential convergence. The chord bounds preserve injectivity. The tangent
bound and the reach inequality
\[
 2\operatorname{dist}(\gamma(t)-\gamma(s),\mathbb R\gamma'(s))
 \le a|\gamma(t)-\gamma(s)|^2
\]
pass to the limit, so its reach is at least $1/a$. Knot type is preserved
by $C^1$ tubular stability about the embedded limit. All other defining
conditions are closed.
\end{proof}

\begin{lemma}[Uniform finite convex interpolation]\label{lem:marked-convex}
Given $a<b$, there is $\delta>0$ such that for any finite family
$\gamma_0,\ldots,\gamma_m\in X_a$ with
$\|\gamma_i-\gamma_0\|_{C^1}<\delta$, every convex combination
$z=\sum_i t_i\gamma_i$ is a regular embedding of type $K$, and $N(z)\in X_b$.
The choice of $\delta$ is independent of $m$.
\end{lemma}
\begin{proof}
One has $1-\delta\le|z'|\le1$, $\|z''\|_\infty\le a$ and $\Len(z)\le1$.
Uniform embeddedness and type follow from the short- and long-chord bounds
of \cref{lem:marked-compact} and small $C^1$ perturbation. We verify the
reach estimate rather than assuming that reach is $C^1$-continuous.
Write $q=z(s+h)-z(s)$ and
$D=\operatorname{dist}(q,\mathbb Rz'(s))$. For $h\le\sqrt\delta$,
\[
 D\le ah^2/2,\qquad |q|\ge h(1-\delta-ah/2),
\]
so $2D/|q|^2\le a+O_a(\sqrt\delta)$.
For $\sqrt\delta\le h\le h_0$, compare with the chord $q_0$ of
$\gamma_0$. Then $|q-q_0|\le\delta h$, $|q_0|\ge h/2$ and the tangent-line
projections differ by $O(\delta)$. Consequently
\[
 D\le (a/2)|q_0|^2+C\delta|q_0|,
 \qquad |q|\ge(1-2\delta)|q_0|.
\]
The resulting ratio is at most $a+O_a(\sqrt\delta)$. For $h\ge h_0$ the
uniform lower bound $|q_0|\ge h_0/2$ gives the same estimate with
$O_a(\delta)$ error. The reach formula therefore yields
$\Rop(z)\le a+O_a(\sqrt\delta)$. Take $\delta$ so this is below $b$.
All estimates apply directly to a convex combination and do not accumulate
an error with the number of vertices. Normalization preserves ropelength.
\end{proof}

\begin{theorem}[Finite factorization across a positive gap]
\label{thm:marked-finite-factorization}
For $a<b$ the inclusion $X_a\hookrightarrow X_b$ factors up to homotopy
through a finite simplicial complex. Hence, over every field, the persistence
module $H_q(X_a)$ is q-tame in every degree $q\ge0$: every map across a
strictly positive parameter gap has finite rank.
\end{theorem}
\begin{proof}
Let $\delta>0$ be given by \cref{lem:marked-convex} for the pair $a<b$ and
put $\rho=\delta/2$. By compactness (\cref{lem:marked-compact}) choose a
finite cover of $X_a$ by open $C^1$ balls $B(x_i,\rho)$ with centres
$x_i\in X_a$, and a subordinate partition of unity $\phi_i$. Let $P$ be the
nerve, $f(x)=\sum_i\phi_i(x)e_i$, and
$g(\sum_i t_ie_i)=N(\sum_i t_ix_i)$. If $\{x_{i_0},\dots,x_{i_k}\}$ spans a
simplex of $P$, the balls have a common point, so
$\|x_{i_j}-x_{i_0}\|_{C^1}<2\rho=\delta$ for all $j$, and
\cref{lem:marked-convex}, applied with $\gamma_0=x_{i_0}$, shows that $g$ is
well defined and continuous into $X_b$. For $x\in X_a$ every centre with
$\phi_i(x)>0$ satisfies $\|x_i-x\|_{C^1}<\rho<\delta$, so the same lemma,
applied with $\gamma_0=x$, shows that
\[
 (x,t)\longmapsto N\bigl((1-t)x+t\sum_i\phi_i(x)x_i\bigr)
\]
is a homotopy in $X_b$. At $t=0$ it is the inclusion, because $N(x)=x$ for
$x\in X_a$, and at $t=1$ it is $gf$. The homology map therefore
factors through the finite-dimensional $H_q(P)$. No good-cover theorem is
needed: the nerve is used only as a finite intermediate complex.
\end{proof}

\begin{remark}
The factorization through the nerve of a fine finite cover is a standard
device for compact metric spaces with a controlled local convex structure.
The input specific to ropelength is \cref{lem:marked-convex}: convex
combinations of nearby thick curves remain thick with a loss that is
uniform in the number of curves, so the local convex structure exists across
every positive filtration gap.
\end{remark}

\begin{corollary}[Additional continuous resources]
On the marked model, impose finitely many $C^1$-continuous resource
inequalities $\Phi_j\le c_j$. Inclusion across a positive ropelength gap
and positive gaps in every additional bound factors through a finite
complex. The analogous strict sublevels have finite-rank maps across such gaps.
\end{corollary}
\begin{proof}
The source is a compact closed subset of $X_a$. Continuity, a finite
neighbourhood cover of that subset and the interpolation homotopy give a
uniform margin for all resource bounds. Apply the same construction. For
strict source bounds, factor through an intermediate closed window.
\end{proof}

The marked thickness-at-least-one, length-at-most-$a$ model with constant
speed and zero centroid has the same homology persistence as $X_a$.
Indeed $n(y)=y/\Len(y)$ is natural in $a$, while $s_a(x)=ax$ is a levelwise
inverse up to homotopy: interpolate the dilation factor from $1$ to
$a/\Len(y)$. Length remains at most $a$ and thickness at least one.
This identification is not an isometry for swept-area costs. Nor is the
finite nerve construction automatically equivariant under phase or rotation;
no coarse-quotient or Borel-model tameness follows without another argument.

\begin{remark}[Why the topology is specified]
Smooth planar polar curves with radius
$1+n^{-3}\cos(n^2\theta)$ converge in $C^1$ to the unit circle, whereas their
maximum curvature grows without bound. Their lengths stay bounded away
from zero and their ropelengths diverge. Thus a weakly compact family need
not lie in a single ropelength sublevel. This does not contradict compact
capture for smooth-family topologies in the earlier sections.
\end{remark}

\subsection{Recovering components rather than only ranks}
For a set $S$, give the vector space $\mathbf k[S]$ its diagonal coalgebra:
$\Delta(s)=s\otimes s$, $\epsilon(s)=1$. Its group-like elements are exactly
the basis elements: if $v=\sum a_ss$, then $\Delta v=v\otimes v$ forces
$a_sa_t=0$ for $s\ne t$, $a_s^2=a_s$, and $\epsilon(v)=1$ leaves one
coefficient equal to one. Consequently the natural coalgebra on singular
$H_0$ recovers the set of path components and all its transition maps.
It does not recover connected components when these differ from path components.

Ordinary barcode data forget incidence. For example, give four leaves
$a,b,c,d$ births $0,1,2,3$. In one tree merge $a,b$ at $4$, $c,d$ at $5$,
and the two pairs at $6$. In another merge $a,b$ at $4$, then their cluster
with $d$ at $5$, and finally with $c$ at $6$. Both have bars
$[0,\infty),[1,4),[3,5),[2,6)$, but the final split has sizes $2+2$ in the
first tree and $3+1$ in the second. Component partitions retain that information.

\subsection{Finite terminal cores}
Suppose an exhausted space $Z$ has compact capture and has the homotopy type
of a finite CW complex $M$. Choose homotopy inverses $u:M\to Z$, $v:Z\to M$.
Compact capture places $u(M)$ in some $Z_A$. Writing $s=u_Av$ gives
$j_As\simeq\mathrm{id}_Z$. Thus the inclusion from a single finite level
surjects onto all terminal homology groups, and, with compatible basepoints,
onto all terminal homotopy groups. This does not make $Z_A\to Z$ an
equivalence: transient classes may persist in its kernel.

\subsection{Links and escape of split factors}
For a fixed non-split link type and nonempty total-length sublevel with
thickness at least one, the normalized weak sublevel is compact. To see the
extra diameter assertion, choose one point $p_i$ on each component and
translate $p_1$ to zero. A bound on total length bounds each component's
diameter. If the full diameter were unbounded, take a subsequence so every
pair $|p_i-p_j|$ either remains bounded or tends to infinity. Bounded
separation is an equivalence relation on the finite component set. The class
of $p_1$ stays in a fixed ball and all other classes escape it. A round sphere
separates the corresponding sublinks for large indices, contradicting
non-splitting. The preceding compactness argument now applies componentwise.
For a split link, translating separated factors apart preserves thickness
and total length, and produces a noncompact sequence. The entry length of a
split union is the sum of the factor entry lengths: restriction gives one
inequality, and distant placement of minimizing factors gives the other.
If that sum is $m$, a total budget $m+e$ allows each factor only its own
entry length plus $e$. Restriction of a link path therefore transfers any
proved factor excess barrier to the split union. This assertion is
conditional on such a factor barrier and supplies no new explicit example.

\subsection{Pole gauges, orbit birth and essential persistence}
For the pole-marked model put
$\mathcal R(J,p)=\Rop(\sigma_pJ)$, with stereographic projection from $p$.
Simultaneous orientation-preserving M\"obius transport preserves this
quantity. Indeed
$\sigma_{\phi p}\phi\sigma_p^{-1}$ is a conformal transformation of the
one-point compactification of $\mathbb R^3$ fixing infinity, hence a Euclidean
similarity, and ropelength is similarity invariant. This is invariance of
the pair $(J,p)$; moving the curve with a fixed pole is a different operation.

Assume the orbit inclusion $SO(4)/G_K\to\mathscr K_K(S^3)$ is a homotopy
equivalence for a specified smooth symmetric placement. Pulling back the
complement bundle along it gives the terminal model
$SO(4)\times_{G_K}(S^3\setminus K)\simeq\mathcal P(K)$: the comparison is
an equivalence on base and on each fibre. Choose a $G_K$-invariant tubular
neighbourhood and write $E_K$ for its compact exterior. Radial retraction
inside the tube is equivariant and identifies $E_K$ with the complement up
to equivariant homotopy. Consequently
\[
 SO(4)\times_{G_K}E_K\longrightarrow\mathcal P(K),\qquad
 [A,p]\longmapsto(AK,Ap)
\]
is a homotopy equivalence whose image lies at a finite level, namely
\[\max_{p\in E_K}\mathcal R(K,p).\] Finiteness follows from compactness and
smooth-family continuity away from the pole. The bound depends on the
placement and tube; no uniform or computable knot-type bound is asserted.

More generally, in a path-connected filtered space on which a group acts
preserving the resource, the terminal image of an orbit map is independent
of the chosen point up to change of basepoint: a path of points gives a
homotopy of orbit maps. If points occur arbitrarily close to the entry
level, every nonzero terminal orbit class has that entry level as its
basepoint-free birth infimum. At a prescribed basepoint this argument does
not pay for a path from that point to the chosen orbit; based widths must
retain that cost. The same argument applies without basepoint choices to
nonzero homology orbit classes.

If a terminal space $Z$ has finite CW type and compact capture, its essential
homology persistence
$E_q(a)=\operatorname{im}(H_q(Z_a;\mathbf k)\to H_q(Z;\mathbf k))$
is an exhaustive nested filtration of a finite-dimensional vector space.
An adapted basis expresses it as finitely many right-infinite intervals,
with open or closed left endpoints according to attainment, and it vanishes
in all sufficiently high degrees. This statement concerns the terminal
image, not the transient part of $H_q(Z_a)$.

For a compact decorated family, the same capture argument works when the
forgetful family and every additional resource function are continuous in
the stipulated topology: each resource has a finite maximum. Transferring
homotopy conclusions additionally requires the topology of the decorated
space, for example an actual bundle or a stated terminal model. A finite
list of decorations does not itself prove such a description.

\section{Problems, graded}
\label{sec:problems}

The subject has few practitioners and a large surface area, so we list
problems by the amount of machinery they need rather than by importance.
Problems in the first group are suitable as a first project; those in the
second require some investment; those in the third are the ones we do not
know how to start.

\subsection*{Entry problems}

\begin{problem}[The unknot just above the minimum]\label{prob:unknot-window}
Describe \(Y_\Lambda(U)\) for \(\Lambda\) slightly larger than \(2\pi\).  Is
it contractible?  Is it connected for every \(\Lambda\)?  Are
\(\Conn(Y_\Lambda(U))\) and \(\pi_0(Y_\Lambda(U))\) finite
(\cref{q:finite-components,prob:tameness})?  A complete answer for
\(\Lambda\le2\pi+\eps\) would already be the first computation of a
nonminimal ropelength sublevel.
\end{problem}

\begin{problem}[Compute an untying profile]\label{prob:profile}
For explicit unknotted thick configurations --- a doubled loop, a coiled
configuration, the configurations of \cite{PieranskiGordian} --- compute
numerically the path min--max profile of \cref{fig:untying}, estimate
\(u_{\mathrm p}(x)\), and seek rigorous lower bounds for the common value
\(u_{\mathrm c}(x)=u_{\mathrm p}(x)\).  Even convincing numerical evidence
that \(e_{\mathrm p}(x)>0\)
for one configuration, together with a rigorous error analysis, would be a
substantial contribution.
\end{problem}

\begin{problem}[A careful ropelength table]\label{prob:table}
Compile a table of the best current ropelength bounds and numerical values
for knots of small crossing number, stated consistently in one normalisation
and with sources and error estimates attached
(cf.\ \cref{conv:normalisation}).  The literature currently mixes the radius
and diameter conventions, which is a recurring source of confusion for
newcomers.
\end{problem}

\begin{problem}[Mirror action in low degree]\label{prob:mirror-low-degree}
For an amphichiral knot type, determine the induced \(\Z/2\)-action of
\cref{prop:mirror-involution} on \(\pi_0(Y_\Lambda(K))\) and on
\(H_*(Y_\Lambda(K);\Z/2)\) in examples, and identify the level at which the
action becomes trivial.
\end{problem}

\subsection*{Intermediate problems}

\begin{problem}[Attainment in positive degree]\label{prob:attainment}
Is the infimum defining the width \(\wid_i^K(\alpha;\gamma_*)\) of
\cref{def:width} attained?  Equivalently, is there a compactness statement
for spaces of thick families analogous to
\cref{thm:knot-sublevel-compact}?  A positive answer would upgrade
\cref{prop:width-properties} to the positive-degree analogue of
\cref{thm:merge-ultrametric}.
\end{problem}

\begin{problem}[Tameness]\label{prob:tameness}
Are the sublevels \(Y_\Lambda(K)\) locally path connected, so that connected
and path components agree (\cref{rem:path-vs-conn})?  Do \(Y_\Lambda(K)\) or
\(\calE_{K,\Lambda}\) have the homotopy type of CW complexes?  Are the
persistent homology modules pointwise finite dimensional over a field?  The marked weak result below settles finite rank across a gap, but not these stronger questions concerning
\cref{fig:barcode}; see \cref{rem:no-barcodes}.  Semialgebraic
approximation by polygons with a fixed number of edges is the natural
approach.
\end{problem}

\begin{remark}[Regularised endpoint connections]\label{prob:swept-regularisation}
The paths constructed in \cref{thm:lpc} are finite concatenations of scaling
and interpolation paths. They provide Lipschitz endpoint connections after
positive right relaxation. This resolves the endpoint existence question;
it makes no assertion about preserving the swept-area cost of a given path.
\end{remark}

\begin{problem}[Produce a transient class]\label{prob:transient}
By \cref{cor:transient-higher} every class in \(\pi_i(\calK_{f,\Lambda})\)
with \(i\ge2\) dies at a finite level.  Exhibit one that is nonzero.
Equivalently: find a sphere of thick long knots that cannot be filled within
its own budget.
\end{problem}

\begin{problem}[Satellite widths]
See \cref{prob:satellite-width}.  This is the first place where
\cref{thm:vanishing-width} does not apply, and therefore the most promising
source of a positive answer to \cref{prob:first-positive-width}.
\end{problem}

\begin{problem}[Comparing the models]\label{prob:models}
Compare the round \(S^3\) filtration, the Euclidean normalised filtration,
and the pole-marked filtration of \cref{sec:bridges}.  Find explicit
interleavings, or conformal gauges, isolating metric-dependent quantitative
data from the topological limit.
\end{problem}

\subsection*{Hard problems}

\begin{problem}[A standard thick Gordian knot]\label{prob:gordian-knot}
In the standard unit-thickness model of this paper, find a knot type \(K\)
--- possibly \(K=U\) --- and a configuration \(x\) of length \(L\) that is
not path connected in \(Y_L(K)\) to a chosen tight target component
\(A_0\).  By \cref{prop:fixed-length-retract} this is exactly a
fixed-length physical obstruction.  A stronger target is to prove
\(u_{\mathrm c}(x)=u_{\mathrm p}(x)>L\), which by
\cref{thm:untying-dichotomy} is an attained connected certificate.  Ayala's Gordian unknots
\cite{AyalaGeometric} occur in the related thinner model \(\mathcal U_1\)
and therefore do not settle this exact formulation.
\end{problem}

\begin{problem}[Disconnected ideal strata]\label{prob:disconnected-ideal}
Prove that \(\Conn(I(K))\) has more than one element for some knot type
\(K\); or prove that \(I(K)\) is always connected.  The symmetric
criticality results of \cite{CEFM} produce distinct \emph{critical}
configurations; distinct \emph{minimising} components are a strictly
stronger requirement.
\end{problem}

\begin{remark}The continuum mirror-separation question is \cref{prob:figure-eight-mirror}. Discrete comparisons are treated in \cite{OzawaDiscrete}.\end{remark}

\begin{problem}[The first positive-degree phenomenon]
See \cref{prob:first-positive-width}.  We repeat it here because it is, in
our view, the question on which the interest of the whole positive-degree
theory rests.
\end{problem}

\section{Concluding perspective}
\label{sec:conclusion}

There are three levels of the same object, and they fit together as in
\cref{fig:threelevels}.  Classical knot theory computes \(\pi_0\) of the
whole embedding space.  Hatcher and Budney compute the homotopy type of one
component.  Ropelength inserts a geometric scale between them, and
\cref{thm:homotopical-exhaustion} proves that nothing is lost in passing to
the limit.

\begin{figure}[t]
\centering
\begin{tikzpicture}[scale=0.86]
\node[draw=ropegrey,fill=ropesand,rounded corners=8pt,inner sep=6pt,
      align=center,font=\small] (a) at (0,0)
  {finite thick-knot\\ geometry};
\node[draw=ropegrey,fill=ropelight,rounded corners=8pt,inner sep=6pt,
      align=center,font=\small] (b) at (5.0,0)
  {ropelength-filtered\\ persistent homotopy};
\node[draw=ropegrey,fill=ropesand,rounded corners=8pt,inner sep=6pt,
      align=center,font=\small] (c) at (10.2,0)
  {Hatcher--Budney\\ homotopy type};
\draw[->,>=Stealth,line width=0.9pt] (a) -- (b);
\draw[->,>=Stealth,line width=0.9pt] (b) -- (c)
  node[midway,above,font=\scriptsize] {$\Lambda\to\infty$};
\node[font=\scriptsize,below] at (2.5,-0.75) {barriers, merges, widths};
\node[font=\scriptsize,below] at (7.6,-0.75)
  {\cref{thm:homotopical-exhaustion}};
\end{tikzpicture}
\caption{The organising diagram of the paper.  The right-hand arrow is a
theorem; the left-hand arrow is where all the open problems live.}
\label{fig:threelevels}
\end{figure}

What the present paper adds to that slogan is a sharper sense of where the
content is.  Compactness of normalised sublevels makes the degree-zero
theory quantitative: barriers, once they exist, are positive and attained,
and merge levels assemble into an ultrametric.  At a fixed level, the
companion swept-area metric \cite{OzawaSwept} adds a second, genuinely
geometric layer by measuring the area cost of regularity-restricted
admissible paths; its current-theoretic separation theorem shows that this
is a metric rather than only a pseudometric.  Symmetry, by contrast, makes
the positive-degree theory collapse in exactly the cases that are best
understood: if a knot space is an orbit of the isometry group, no homotopy
class costs any rope.  The two statements point in opposite directions and
together they say that the interesting phenomena should be sought in
degree zero, and for stable positive-degree width most naturally in directions not generated
by an ambient isometry orbit, while transient classes remain possible even in orbit-model cases.

The central long-term question is therefore not how short a knot can be.  It
is how the topology of the space of \emph{all} realisations of a knot is
assembled as geometric freedom increases, and which parts of that assembly
admit finite certificates.  The subject is unusual in having, at the moment,
more well-posed elementary questions than practitioners; we hope that the
list in \cref{sec:problems} is useful to anyone who would like to change
that.

\appendix

\section{Notation}
\label{app:notation}

\begin{center}
\small
\begin{tabular}{@{}ll@{}}
\toprule
\(\Len,\Thi,\Rop\) & length, thickness (tube radius), ropelength
  \(=\Len/\Thi\)\\
\(\operatorname{dcsd}\) & doubly critical self-distance\\
\(m_K=\Rop(K)\) & minimal ropelength of the knot type \(K\)\\
\(\calE_K(M)\) & component of \(\Emb^\infty(S^1,M)\) of type \(K\)\\
\(\calE_{K,\Lambda}(M)\) & its ropelength sublevel \(\{\Rop\le\Lambda\}\)\\
\(\wh\calE_{K,\Lambda}(M)\) & the \(C^{1,1}\) positive-reach sublevel\\
\(\calU_K,\mathscr K_K\) & oriented and unoriented unparameterised knot
  spaces\\
\(\calK,\calK_f\) & space of long knots and one of its components\\
\(Y_\Lambda(K)\) & \(\{\Thi\ge1,\Len\le\Lambda\}\) modulo rigid motions\\
\(I(K)=Y_{m_K}(K)\) & ideal stratum\\
\(\Conn(X),\pi_0(X)\) & connected components and path components, respectively\\
\(\mu_K,\mu_K^{\rm path}\) & connected and path merge levels\\
\(d_K=\mu_K-m_K\) & attained ultrametric on connected components of \(I(K)\)\\
\(d^{\mathrm{SA}}_\Lambda\) & swept-area extended metric of \cite{OzawaSwept}\\
\(u_{\rm c},e_{\rm c};\ u_{\rm p},e_{\rm p}\) & connected and path-budget untying thresholds/excesses\\
\(\wid_i^K(\alpha),\exc_i^K(\alpha)\) & ropelength width and excess width\\
\(\sigma_i(K)\) & stabilisation level in degree \(i\)\\
\(\tau\) & mirror involution\\
\(\calP(K),\calR\) & pole-marked space and stereographic ropelength\\
\bottomrule
\end{tabular}
\end{center}

\section{An annotated guide to the literature}
\label{app:literature}

This appendix is meant for a reader entering the subject.  It is short by
design.

\smallskip\noindent\textbf{Thickness and ropelength: the basics.}  Start with
Cantarella--Kusner--Sullivan \cite{CKS}: existence and \(C^{1,1}\)
regularity of minimisers, the thickness formula, and the first serious lower
bounds.  Gonzalez--Maddocks \cite{GonzalezMaddocks} give the three-point
characterisation \eqref{eq:three-point}, which is the form of the definition
one actually computes with; Litherland--Simon--Durumeric--Rawdon \cite{LSDR}
is the companion topological account, and Federer \cite{Federer} is the
origin of ``reach''.  The volume \cite{IdealKnots} collects the early
physical and numerical literature.

\smallskip\noindent\textbf{Bounds and numbers.}  F\'ary \cite{Fary} and
Milnor \cite{Milnor} give the elementary bound of
\cref{prop:farymilnor-bound}; Diao \cite{Diao} and then Denne--Diao--Sullivan
\cite{DDS} give the current record.  Buck--Simon \cite{BuckSimon} relate
ropelength to crossing number.  For numerics, Ashton--Cantarella--Piatek--Rawdon
\cite{ACPR} is the standard reference, and \cite{Ropelength12} is a recent
large-scale tabulation.

\smallskip\noindent\textbf{Criticality and symmetry.}  Cantarella--Fu--Kusner--Sullivan
\cite{CFKS} develop the first-order theory of ropelength.
Cantarella--Ellis--Fu--Mastin \cite{CEFM} prove symmetric criticality and, in
particular, exhibit knot types with several non-congruent critical
configurations; this is the paper to read before attacking
\cref{prob:disconnected-ideal}.

\smallskip\noindent\textbf{Gordian phenomena and the importance of the model.}
Pieranski--Przyby\l{}--Stasiak \cite{PieranskiGordian} is the numerical
origin of the Gordian-unknot question.  Coward--Hass \cite{CowardHass} is
the foundational theorem separating topological and physical link theory;
Kusner--Kusner \cite{KusnerKusner} allow a total-length constraint.
Ayala--Hass \cite{AyalaHass} give broad families of Gordian unlinks, while
Bauermeister \cite{Bauermeister} treats the more permissive Gehring setting.
Ayala \cite{AyalaGeometric} proves Gordian unknots in a thinner model and is
essential reading precisely because it shows how sensitively the answer
depends on the relation between tube radius and curvature bound.

\smallskip\noindent\textbf{Spaces of knots.}  Gramain \cite{Gramain} is
short and is the right starting point.  Hatcher's proof of the Smale
conjecture \cite{HatcherSmale} underlies everything; \cite{HatcherSpaces}
and the unfinished manuscript \cite{HatcherModuli} contain the models used in
\cref{thm:hatcher-models}.  Budney \cite{BudneyLittleCubes,BudneyTopology}
gives the operadic and recursive structure of long knot spaces.  Queffelec
\cite{Queffelec} is a useful modern reference for the transversality route
to Reidemeister's theorem.  Sinha
\cite{Sinha} is a readable entry point to the higher-dimensional and
cosimplicial side.  For the group-action input see Dinkelbach--Leeb
\cite{DinkelbachLeeb}.

\smallskip\noindent\textbf{Knot energies, for contrast.}  O'Hara
\cite{OHara} and Freedman--He--Wang \cite{FHW} study a different functional
on the same space of curves.  Much of the framework of this paper applies
verbatim to any functional that is locally bounded on the embedding space
and has compact normalised sublevels; comparing the two landscapes is an
attractive project.

\smallskip\noindent\textbf{Discrete and lattice models.}  The BFACF move set
originates in \cite{BergFoerster,ACF}; \cite{OzawaDiscrete} gives the complete
minimal-layer merge-tree computations, and
\cite{OzawaFinite,OzawaSwept} are the companion papers on filtered
Reidemeister graphs and on swept-area metrics.  The latter proves
non-degeneracy by a flat-current lower bound and develops projected-area
calibrations and swept-area costs for admissible merge problems.  Rawdon
\cite{RawdonApprox} is the standard reference on polygonal approximation of
thickness, and is the analytic input any lattice-to-continuum comparison
will need.

\subsection*{Declaration of generative AI and AI-assisted technologies}

During the preparation and revision of this manuscript the author used
ChatGPT (OpenAI) and Claude (Anthropic) as tools for drafting, restructuring,
language refinement, figure prototyping, and preliminary consistency checks.
The author reviewed and edited all AI-assisted material, checked the mathematical statements, proofs, computations and references, and takes full responsibility for the content of this manuscript.

\end{document}